\documentclass[11pt]{amsart}
\usepackage{amsmath}
\usepackage{amsfonts}
\usepackage{amssymb}
\usepackage{amscd}
\usepackage[arrow,matrix,curve,cmtip,ps]{xy}
\usepackage{graphicx}

\usepackage{framed,color}
\definecolor{shadecolor}{rgb}{1,0.8,0.3}

\numberwithin{equation}{section}

\usepackage{amsthm}

\theoremstyle{plain}
\newtheorem{thm}{Theorem}[section] 
\newtheorem{theorem}[thm]{Theorem}
\newtheorem{proposition}[thm]{Proposition}
\newtheorem{corollary}[thm]{Corollary}
\newtheorem{lemma}[thm]{Lemma}

\theoremstyle{definition}

\newtheorem{remark}[thm]{Remark}

\theoremstyle{remark}
\newtheorem{example}[thm]{Example}

\newcommand\id{\mathop{\rm id}}

\newcommand{\cl}[1]{\mathcal{#1}}

\newcommand{\bb}[1]{\mathbb{#1}}

\begin{document}

\title{Relative Weak Injectivity for C*-algebras}
\author{Ali S. Kavruk}
\thanks{2010 Mathematics Subject Classification. Primary 46L06, 46L07; Secondary 46L05, 47L25, 47L90.}
\thanks{Key words. operator system, tensor product, relative weak injectivity, Arveson Riesz property, Connes' embedding problem.}
\begin{abstract}

We study the order theoretic properties of relative weak injectivity, w.r.i., in short, in the category of C*-algebras. We prove that Arveson's extension theorem, with additional order assumption on the morphisms, is tightly connected with relative weak injectivity. We prove that (2,3)-Riesz-Arveson property, defined below, is equivalent to w.r.i. Likewise tight Riesz interpolation property yields another characterization of w.r.i. We exhibit, with examples, that these C*-algebraic properties fail in general operator systems. Several order theoretic characterization of Connes' embedding problem is given.

\end{abstract}

\maketitle

C$.$ Lance's weak expectation property characterises C*-algebras for which the maximal tensor product behaves injectively \cite{Lance2}, \cite{Lance}. For von Neumann algebras this property coincides with the injectivity and, as pointed out by Effros and Haagerup, for general C*-algebras, this may be viewed as the approximate injectivity for matrix systems \cite{Effros-Haagerup}. Kirchberg's QWEP conjecture states that every C*-algebra is a quotient of a C*-algebra with WEP, which is equivalent to Connes' embedding problem on the embedding of II$_1$-factors into ultrapowers of hyperfinite II$_1$-factor \cite{Kirchberg94}, \cite{Connes}. A C*-algebra $\cl A$ with WEP is also called {\it weakly injective} C*-algebra as the bidual von Neumann algebra $\cl A^{**}$ of a $\cl A$ is injective relative to $\cl A$ \cite{Kirchberg-Presentation}, \cite{Paulsen WEP}. WEP is a categorical concept and indigenous to non-commutative algebra, in fact, in classical function systems (such as Banach spaces or Kadison spaces) WEP-like properties coincide with the categorical nuclearity. This property extensively studied in operator spaces \cite{pisier_intr}, \cite{Haagerup}, non-selfadjoint algebras \cite{Blecher}, operator systems \cite{KPTT Nuclearity} and Hilbert C*-modules \cite{JIAN LIANG}.

\smallskip

A more general concept, namely relative weak injectivity or weak relative injectivity, w.r.i$.$ in short, is introduced by Kirchberg in \cite{Kirchberg94}. For C*-algebras $\cl A \subseteq \cl B$, $\cl A$ is called {\it w.r.i}$.$ in $\cl B$ if the inclusion of $\cl A$ in $\cl B$ admits a conditional expectation from $\cl B$ into $\cl A^{**}$. (Throughout we assume that all the C*-algebras are unital and all C*-algebraic inclusions are unital.) If the larger object $\cl B$ is injective this property coincides with WEP. W.r.i$.$ is a nuclearity related property and characterizes the pair of C*-algebras $\cl A \subseteq \cl B$ for which the projective tensor product behaves injectively in the sense that $\cl A \otimes_{\max} \cl C \subseteq \cl B \otimes_{\max} \cl C$ for every C*-algebra $\cl C$. Moreover, by the standard Lance trick, the later condition can be solely verified by $\cl C =C^*(\bb F_{\infty})$, the (full) group C*-algebra of free group $\bb F_{\infty}$ generated by countably infinite number of generators.

\smallskip

Our main purpose in this paper is to illuminate the connection between relative weak injectivity and classical order theoretic concepts such as tight Riesz interpolation or Riesz decomposition properties. In addition to the non-commutative order of the underlying operator algebras, there is also a non-commutative order on the morphism. More precisely, for completely bounded self-adjoint maps $\varphi_1$ and $\varphi_2$ we shall write $\varphi_1 \leq \varphi_2$ if $\varphi_2- \varphi_1$ is a completely positive map. Arveson's theory of Radon-Nykodim derivatives prescribes the structure of uniformly dominated completely positive maps \cite{Arveson}. Here we shall be concerned with Arveson's extension theory with additional order preservation assumptions on the morphisms, which in turn, is tightly connected with the relative weak injectivity. As a first step we prove the following:

$ $

\begin{theorem}\label{22 RA Prop} Let $\cl C$ be an injective C*-algebra. For $\cl A \subseteq \cl B$ the following hold:
\begin{enumerate}
\item every completely positive maps $\varphi_1, \varphi_2, \varphi : \cl A \rightarrow \cl C$ with $\varphi_1, \varphi_2 \leq \varphi$ have completely positive extensions
 $\tilde \varphi_1, \tilde \varphi_2, \tilde \varphi : \cl B \rightarrow \cl C$, resp.,  with
 $\tilde \varphi_1, \tilde \varphi_2 \leq \tilde \varphi;$
 
 \smallskip
 
 \item every completely positive maps $\varphi_1, \varphi_2, \varphi_3, \varphi_4: \cl A \rightarrow \cl C$ with $\varphi_1 + \varphi_2 = \varphi_3 + \varphi_4$ admit completely positive extensions 
  $\tilde \varphi_1, \tilde \varphi_2, \tilde \varphi_3, \tilde \varphi_4: \cl B \rightarrow \cl C$, resp., with $\tilde \varphi_1 + \tilde \varphi_2  = \tilde \varphi_3 + \tilde \varphi_4.$
\end{enumerate}
\end{theorem}

The condition in (2) will be called \textit{$(2,2)$-Riesz-Arveson property} of the extensions.
It is worth mentioning that these C*-algebraic properties may fail in general operator systems. In example \ref{1. example} we will see that the Namioka and Phelps test subsystem does not possess these properties. 
A simple generalization of the above theorem requires an additional assumption, in the following the property in (3) will be called (2,3)-\textit{Riesz-Arveson property}.

\begin{theorem} The following statements are equivalent for $\cl A \subseteq \cl B$:
\begin{enumerate}
 \item $\cl A$ is w.r.i.\! in $\cl B$;
 
 \smallskip
 
 \item for all $m$, every completely positive maps $\varphi_1, \varphi_2, \varphi_2, \varphi : \cl A \rightarrow M_m$ with $\varphi_1, \varphi_2, \varphi_3 \leq \varphi$ have completely positive extensions 
  $\tilde \varphi_1, \tilde \varphi_2, \tilde \varphi_3, \tilde \varphi : \cl B \rightarrow M_m$ with $\tilde \varphi_1, \tilde \varphi_2, \tilde \varphi_3 \leq \tilde \varphi$;
 
 \smallskip
 
 \item for every $m$, every completely positive maps $\varphi_i : \cl A \rightarrow M_m$, $i=1,...,5$, with 
 $$
 \varphi_1 + \varphi_2 = \varphi_3 + \varphi_4 + \varphi_5
 $$ 
 have completely positive extensions $\tilde \varphi_i : \cl B \rightarrow M_m$, $i=1,...,5$, resp., with 
 $$
 \tilde \varphi_1 + \tilde \varphi_2 = \tilde \varphi_3 + \tilde \varphi_4 + \tilde \varphi_5;
$$

\smallskip

\item for every m, every completely positive maps $\varphi_i: \cl A \rightarrow M_m$, $i=1,...,6$, with
$$
\varphi_1 +\varphi_2 = \varphi_3 + \varphi_4 = \varphi_5 + \varphi_6
$$
have completely positive extensions $\tilde \varphi_i: \cl B \rightarrow M_m$, $i=1,...,6$, resp., with
$$
\tilde \varphi_1 + \tilde \varphi_2 = \tilde \varphi_3 + \tilde \varphi_4 = \tilde \varphi_5 + \tilde \varphi_6.
$$  
\end{enumerate}
\end{theorem}

We fix a pair of C*-algebras $\cl A \subseteq \cl B$ where $\cl A$ is a unital C*-algebra of $\cl B$. For self-adjoint elements $a_1, a_2$ we shall write $a_1 < a_2$ if $a_2 - a_1 \geq \delta 1$ for some scalar $\delta>0$. We say that $\cl A$ has the {\it $(n,k)$-tight Riesz interpolation property in} $\cl B$, TR($n,k$)-property in short, if the following holds: for every selection of self-adjoint elements $x_1,...,x_n$ and $y_1,...,y_k$ of $\cl A$, if $x_1,...,x_n$ and $y_1,...,y_k$ interpolate in $\cl B$, meaning that there exists a self-adjoint element $b$ of $\cl B$ such that
$$
x_1,...,x_n < b < y_1,...,y_k,
$$
then $x_1,...,x_n$ and $y_1,...,y_k$ interpolate in $\cl A$ (i.e $x_i < a <y_i$ for all $i=1,...,n$, $j=1,...,k$ for some $a \in \cl A_{sa}$). Likewise we say that $\cl A$ has the {\it complete} TR($n,k$){\it-property in} $\cl B$ if $M_m(\cl A)$ has the TR(n,k)-property in $M_m(\cl B)$ for every $m$.

\smallskip

The above definition should be viewed as a non-commutative version of classical Riesz interpolations considered in ordered groups, Kadison spaces, Riesz spaces etc$.$ \cite{Teller} , \cite{NP}, \cite{AT}. Paulsen's example \cite[Sec$.$ 7]{kavruk2012} indicates that interpolation in a larger object is essential. When the larger object $\cl B$ is injective, several WEP-related formulae in terms of Riesz interpolation are obtained in \cite{kavruk2012}. (See also \cite{GS} for a model theoretic approach.) As a first step we will prove that:

\begin{theorem}For any pair $\cl A \subseteq \cl B$ where $\cl A$ is a unital C*-subalgebra of $\cl B$, $\cl A$ has TR(2,2)-property in $\cl B$.
\end{theorem}

In Example \ref{2. example} we will see that this property may fail for general operator systems. The following improves the results in \cite{kavruk2012}:

\begin{theorem}
The following are equivalent for C*-algebras $\cl A \subseteq \cl B:$

\begin{enumerate}

\item $\cl A$ is w.r.i. in $\cl B$;

\item $\cl A$ has the complete TR(2,3)-property in $\cl B$;

\item $\cl A$ has the complete TR(n,k)-property in $\cl B$ for any $n$ and $k$.

\end{enumerate}

\end{theorem}

Unfortunately we are unable to show if ``complete" can be removed in (2) or (3). However, there are special cases where TR$(n,k)$-property implies complete TR$(n,k)$-property. For example, if we fix an embedding of $C^*(\bb F_{\infty}) $ into a $ B(\cl H)$ then $C^*(\bb F_{\infty})$ has TR(2,3)-property in $B(\cl H)$ if and only if it has the complete TR(2,3)-property in 
$B(\cl H)$. In the following equivalence (1) and (2) improves results in \cite{kavruk2012}:

\begin{theorem} We fix a representation $C^*(\bb F_\infty) \subset B(\cl H)$.  The following are equivalent:

\begin{enumerate} 

\item Connes' embedding problem has an affirmative solution;

\smallskip

\item $C^*(\bb F_\infty)$ has  TR(2,3)-property in $B(\cl H)$;

\smallskip

\item  every positive linear functionals $\varphi_1, \varphi_2, \varphi_3, \varphi : C^*(\bb F_\infty) \rightarrow \bb C$ with $\varphi_1, \varphi_2, \varphi_3 \leq \varphi$ have positive extensions 
  $\tilde \varphi_1, \tilde \varphi_2, \tilde \varphi_3, \tilde \varphi : B(\cl H) \rightarrow \bb C$ with $\tilde \varphi_1, \tilde \varphi_2, \tilde \varphi_3 \leq \tilde \varphi$;

\smallskip

\item every positive linear functionals $\varphi_i : C^*(\bb F_\infty) \rightarrow \bb C$, $i=1,...,5$, with 
$$
\varphi_1 + \varphi_2 = \varphi_3 + \varphi_4 + \varphi_5
$$ 
have positive extensions  $\tilde \varphi_i: B(\cl H) \rightarrow \bb C$, $i=1,...,5$, resp., with $$
\tilde\varphi_1 + \tilde\varphi_2 = \tilde\varphi_3 + \tilde\varphi_4 + \tilde\varphi_5;
$$ 

\item every positive linear functionals $\varphi_i : C^*(\bb F_\infty) \rightarrow \bb C$, $i=1,...,6$, with 
$$
\varphi_1 + \varphi_2 = \varphi_3 + \varphi_4 = \varphi_5 + \varphi_6
$$ 
have positive extensions $\tilde \varphi_i: B(\cl H) \rightarrow \bb C$, $i=1,...,6$, resp., with $$
\tilde\varphi_1 + \tilde\varphi_2 = \tilde\varphi_3 + \tilde\varphi_4 = \tilde\varphi_5 + \tilde\varphi_6;
$$
\item every cp maps $\varphi, \psi : C^*(\bb F_\infty) \rightarrow M_3$ with $tr\circ \varphi = tr \circ \psi$ admit cp extensions $\tilde \varphi$, $\tilde \psi$ on $B(\cl H)$ with $tr\circ \tilde\varphi = tr \circ \tilde\psi$, where $tr$ denotes the trace.
\end{enumerate}
\end{theorem}

Since the above formulations are independent of a particular embedding if, in particular, we fix a representation  $C^*(\bb F_\infty) \subset B(\cl H)$ in the sense of
Blackadar \cite{Blackadar} (so the inclusion admits a weak expectation on $B(\cl H)$ into the bi-commutant of $C^*(\bb F_\infty)$) then for every 
$$x_1,x_2 < b < y_1,y_2,y_3,
$$
where $x_i, y_j \in C^*(\bb F_\infty)$ for $i=1,2$, $j=1,2,3$ and $b\in B(\cl H)$ and for every finite dimensional Hilbert subspace $\cl H_0 \subset \cl H$ there exists a self-adjoint $a \in  C^*(\bb F_\infty)$ such that that
$$
P_{\cl H_0} x_i P_{\cl H_0} <
P_{\cl H_0} a P_{\cl H_0} < P_{\cl H_0} y_j P_{\cl H_0}, \;\; i=1,2, \; j=1,2,3,
$$
where $P_{\cl H_0}$ is the orthogonal projection onto $\cl H_0$.
\smallskip

\textit{Remark:} Consider the inclusions 
$C^*(\bb F_2) \subset B(\cl H)$, $C^*(\bb F_2) \subset \Pi M_{n(k)}$, $C^*(\bb F_\infty) \subset \Pi M_{n(k)}$ (where second and third inclusions follow from free groups being residually finite dimensional \cite[Chp$.$ 7]{BO}), $C^*(SL(2,\bb Z)) \subset B(\cl H)$, $C^*(G_1 * G_2) \subset B(\cl H)$, where $G_1$, $G_2$ are discrete abelian groups satisfying $|G_1|\geq 2$, $|G_2|\geq 2$ with $|G_1| + |G_2| \geq 5$. Then Connes' embedding problem is affirmative if and only if one of these inclusions have TR(2,3)-property (if and only if all the inclusions have the complete TR(2,3)-property).
Likewise Connes' embedding problem has an affirmative solution if and only one of these inclusions satisfies one of conditions in the above theorem.

\smallskip

An operator subsystem $\cl S \subset \cl B$, where $\cl B$ is a  C*-algebra, is said to have \textit{C*TR(n,k)-property} in $\cl B$ if every selection of $n$ and $k$ self-adjoint elements in $\cl S$ that interpolate in $\cl B$ also interpolate by a self-adjoint  element in $C^*\{\cl S\}$, the C*-subalgebra generated by $\cl S$ in $\cl B$. We say that $\cl S$ has  \textit{complete C*TR(n,k)-property} in $\cl B$ if $M_m(\cl S) \subset M_m(\cl B)$ has \textit{C*TR(n,k)-property} for every $m$.

\begin{theorem} Let $\cl S$ be an operator subsystem of the C*-algebra $\cl B$. Assume that $\cl S$ is a unitary operator subsystem of  $C^*\{\cl S\}$, the C*-algebra generated by $\cl S$ in $\cl B$. Then $\cl S$ has the complete C*TR(n,k)-property in $\cl B$ if and only if $C^*\{\cl S\}$ has complete TR(n,k)-property in $\cl B$.
\end{theorem}

We consider the category of operator systems with fixed states, denoted by $(\cl S,w)$, which may be viewed as non-commutative probability spaces. Our proof uses categorical pullback of these objects. As in the C*-algebra category \cite{pedersen} we define the \textit{pullback} of a collection $\{(\cl S_i, w_i)\}_{i\in I}$ as the operator system
$$
\sqcap_{i} \cl S_i  = \{ (s_i ) \in \Pi_{i} S_i: w_i(s_i) = w_j(s_j) \mbox{ for all } i,j \in I\}
$$
along with the fixed state $w((s_i)) = w_i(s_i)$ for some $i \in I$. The pullback $(\sqcap_{i} \cl S_i,w)$ has the universal property that for every operator system $\cl T$ and completely positive maps $\varphi_i: \cl T \rightarrow \cl S_i$, $i\in I$, with $w_i \circ \varphi = w_j \circ \varphi$ for all $i,j \in I$, there is a unique ucp map $\varphi: \cl T \rightarrow \sqcap \cl S_i$ such that $\varphi \circ Q_i= \varphi_i$ where $Q_i$ is the canonical quotient map from $\sqcap \cl S_i$ onto $\cl S_i$. In particular, for $I = \{1,2\}$ we have the commuting diagram
$$
\xymatrix{
 \cl T \ar[rdd]_{\varphi_2} \ar[rrd]^{\varphi_1}  \ar@{.>}[rd]|-{\varphi} & &   \\
  & \cl S_1\sqcap \cl S_2 \ar[rd]|-{w} \ar@{->>}[r]_{Q_1} \ar@{->>}[d]^{Q_2}   & \cl S_1 \ar[d]^{w_1} \\
 & \cl S_2 \ar[r]_{w_2} & \bb C
}.
$$

Categorical pullback is the dual notion of categorical pushout, or coproducts, of operator systems and have many nuclearity-related stability properties. If we consider the inclusion of $M_2 \sqcap M_3 \subset M_2 \oplus M_3$ via normalized tracial states, then the question whether or not every ucp map from $M_2 \sqcap M_3$ to $C^*(\bb F_\infty)$ has an approximate ucp extension on $M_2 \oplus M_3$ would be equivalent to Connes' embedding problem (follows from Theorem 0.6 of \cite{Kavruk NP}).

\smallskip

We shall be particularly interested in the operator system $\ell_n^{\infty}$ with the fixed state $w$ given by
$
w((a_1,...,a_n)) = (a_1+\cdots + a_n)/n. 
$
Note that $\ell_n^\infty \sqcap \ell_n^\infty$ can be identified with  the operator subsystem
$$
\cl V = \{ (a_1,a_2,...,a_{2n}): \sum_{i=1}^n a_i = \sum_{i=n+1}^{2n} a_i \} \subseteq \ell_{2n}^\infty.
$$
If $\varphi_i $ and $ \psi_i$, $i=1,...,n$ are positive linear functionals on a C*-algebra $\cl A$ with the property that
$$
\varphi_1 + \cdots +\varphi_n = \psi_1 + \cdots +\psi_n
$$
then we obtain a well-defined completely positive map from $\cl A$ to 
$
\ell_n^\infty \sqcap \ell_n^\infty
$
given by $$a\mapsto (\varphi_1(a), ...., \varphi_n(a), \psi_1(a),...,\psi_n(a)).$$
Conversely every completely positive map $\varphi: \cl A \rightarrow \ell_n^\infty \sqcap \ell_n^\infty$ yields positive linear functionals with the above property when evaluated component-wise. We shall use this argument repeatedly in the proofs.

$ $

We start with a brief introduction to operator systems and tensor products. We then outline the duality correspondence of pullbacks and pushouts. We prove, for unital C*-algebras $\{\cl A_i\}$, that the canonical inclusion of pushout $\sqcup_i \cl A_i \subseteq *_i \cl A_i$ constitutes an essential unitary operator system.  (In this paper we use the notation $\sqcup$ rather than $\oplus_1$ for pushouts.) We recall w.r.i$.$ characterizations of Kirchberg then obtain further formulations. In the later section we focus on the proof of the main results and some examples in operator systems.


\section{Preliminaries}

In this section we establish the terminology and state basic definitions and results that
shall be used throughout the paper. An operator system can be defined concretely as a unital $*$-closed subspace of $B(\cl H)$, bounded linear
transformations acting on a Hilbert space $\cl H$. We refer the reader to \cite{PaulsenBook} for
the abstract characterization of these objects via compatible, strict collection of matricial cones
along with an Archimedean matrix order unit \cite{ChoiEffros77}. A map between operator systems
$ \varphi: \cl S \rightarrow \cl T$ is called completely positive, cp in short, if the $n^{\rm th}$-amplification
$id_n \otimes \varphi: M_n \otimes \cl S \rightarrow M_n \otimes \cl T$ is positive for all $n$.
If $\varphi$ is also unital, i.e. $\varphi(1_{\cl S}) = 1_{\cl T}$, we will say that $\varphi$ is a ucp map.
$\mathfrak{S}_n(\cl S) = \{ \varphi: \cl S \rightarrow M_n: \; \varphi \mbox{ is ucp } \}$ denotes the $n^{\rm th}$ matricial state space of $\cl S$ and CP($\cl S, \cl T$) denotes the cp maps from $\cl S$ to $\cl T$.

\subsection{Duality} The topological dual $\cl S^*$ of an operator system $\cl S$ has a canonical matricial
order structure: first by defining the self-adjoint idempotent $*$ via $f^*(s) = \overline{f(s^*)}$, we declare
$$
(f_{ij}) \in M_n(\cl S^*) \mbox{ is positive if } \cl S \ni s \longmapsto (f_{ij}(s)) \in M_n \mbox{ is a cp map}. 
$$
The collection of the cones of the positive elements $\{M_n(\cl S^*)^+\}_{n=1}^\infty$ forms a strict, compatible matricial order structure on $\cl S^*$.
In general an Archimedean matrix order may fail to exist for this matricially ordered space. If dim$(\cl S) < \infty$ then a faithful state $w$ on $\cl S$ can be assigned as an Archimedean order unit for $\cl S^*$ \cite{ChoiEffros77}.

\smallskip

\textit{Remark.} There are many infinite dimensional operator system $\cl S$ such that $\cl S^*$ has a structure of an operator system, in other words, there exists a state $w$ on $\cl S$ that uniformly dominates all the other positive linear functionals. Recall that an operator space $X$ can be defined concretely as a subspace of $B(\cl H)$ along with the matricial norm structure. Associated to $X$ let $\cl P_X$ be the universal Paulsen system, which can be described concretely by
$$
\cl P_X  = \mbox{span}\{ \left[\begin{array}{ccccccc} 
\ddots & \ddots & \ddots &&&& \\
& y^* & \lambda I & x & 0 & 0 & \cdots \\ 
\cdots & 0 & y^* &  \underline{\lambda I} & x  & 0 & \cdots \\
\cdots  & 0 & 0 & y^* & \lambda I & x  &  \\
 &  & &  & \ddots & \ddots & \ddots
\end{array} \right] : \lambda \in \bb C,\; x, y \in X \} \subseteq B\left(\bigoplus_{i\in \bb Z} \cl H\right)
$$
Note that the super-diagonal inclusion $\xymatrix{
X \ar@{^{(}->}[r]^{i}  & \cl P_X}$ is a complete isometry.  $\cl P_X$ has the following universal property: for every operator system $\cl T$, for every completely contractive map $\varphi: X \rightarrow \cl T$ there exists a unique ucp map $\tilde \varphi: \cl P_X \rightarrow \cl T $ such that $\tilde \varphi \circ i = \varphi$.  (Note that the universal property of $\cl P_X$ is also sufficient to define itself.) We claim that matricially ordered space $\cl P_X^*$ is an operator system if we declare the state $w: \cl P_X \rightarrow \bb C$ that extends $X\ni x \mapsto 0 \in \bb C$ an Archimedean order unit. In fact given a state $\varphi $ on $ \cl P_X$, let $\varphi_0$ be its restriction to $X$. Note that $\varphi_0$ and $-\varphi_0$ are both c.c. maps and that $\varphi + \widetilde{-\varphi_0} = 2 w$, where $\widetilde{-\varphi_0}$ is the state extending $-\varphi_0$. This, in particular, means that $2w$ dominates $\varphi$. So $\cl P_X^*$ is an operator system with unit $w$.

\subsection{Quotients} A subspace $J \subset \cl S$ is called a \textit{kernel} if $J$
is kernel of a ucp map defined from $\cl S$ (equivalently kernel of a cp map).
A kernel is typically a non-unital $*$-closed subspace but these properties, in general, do not characterise a kernel.
A matricial order structure on the algebraic quotient $\cl S/J$ can be defined
by
$$
Q_n = \{ (s_{ij} + J): (s_{ij}) \in M_n(\cl S)^+) \}.
$$
The Archimedeanization process, in other words, completion of the
cones $\{Q_n\}$ relative to order topology induced by $(e +J) \otimes I_n$ (see \cite{Paulsen-Tomforde}, \cite{KPTT Tensor})
yields the operator system quotient $\cl S/J$.
The universal property of the quotient ensures that if $\varphi: \cl S \rightarrow \cl T$
is a ucp map then the induced map $\dot{\varphi}: \cl S/{\rm ker}(\varphi) \rightarrow \cl T$ is again a ucp map \cite{farenick--paulsen2011}.
$\varphi$ is called a \textit{quotient}  (resp.,\textit{ complete quotient}) map if $\dot{\varphi}$ is an order (resp. a complete order)  inclusion.
In particular a surjective ucp map $\varphi$ is completely quotient if and only if the adjoint $\varphi^\dag: \cl T^* \rightarrow \cl S^*$ is a complete order inclusion. We finally remark that a closed, $*$-closed subspace $J$ of an operator system which does not include any positive elements other than $0$ is called a \textit{null-subspace}. In [25] it was proven that every finite dimensional null-subspace is a kernel. (Interestingly we don't know if every null-subspace is a kernel or not.) Consider the operator spaces $X\subseteq Y$. It follows that $Y+Y^*$ is a kernel in $\cl P_X$, the universal Paulsen system of $X$, such that we have canonical complete order isomorphism $\cl P_X / (Y+Y^*) = \cl P_{X/Y}$.

\subsection{Minimal tensor product} For operator systems $\cl S$ and $\cl T$ we define
$$
C_n^{\min} = \{ [x_{ij}] \in M_n(\cl S \otimes \cl T): \; [(\phi \otimes \psi)(x_{ij})] \geq 0 \;
\forall \phi \in \mathfrak{S}_p(\cl S),\; \psi \in \mathfrak{S}_q(\cl T),\; \forall\; p,q  \}.
$$
The collection of cones $\{C_n^{\min}\}_{n=1}^\infty$ forms a strict compatible matricial ordering for the algebraic tensor $\cl S \otimes \cl T$.  
Moreover, $1_{\cl S} \otimes 1_{\cl T}$ is a  Archimedean matricial order unit. Therefore the triplet
$
(\cl S \otimes \cl T, \{C_n^{\min}\}_{n=1}^\infty, 1_{\cl S} \otimes 1_{\cl T})
$
forms an operator system which we call the minimal tensor product of $\cl S$ and $\cl T$ and denote by $\cl S\otimes_{\min} \cl T$. 
We refer the reader to \cite{KPTT Tensor} for details. The minimal tensor product is spatial, injective and functorial. By the representation
of the minimal tensor we mean
CP$(\cl S, \cl T) \cong (\cl S^*\otimes_{\min} \cl T)^+
$
for any operator systems $\cl S$ and $\cl T$ with dim$(\cl S)<\infty$ \cite{kavruk2011}.

\subsection{Maximal tensor product} Let $\cl S$ and $\cl T$ be two operator systems. We define
$$
D_n^{\max} = \{ X^* (S \otimes T) X:  \; S \in M_p(\cl S)^+, \; T \in M_q(\cl T)^+,\; X \mbox{ is }pq\times n \mbox{ matrix}, \; p,q \in \bb N \}.
$$
The collection of the cones $\{D_n^{\max}\}_{n=1}^\infty$ are strict and compatible. Moreover, $1\otimes 1$ is a matricial order unit
for the matrix ordered space $(\cl S \otimes \cl T, \{D_n^{\max}\})$. Nonetheless $1\otimes 1$ may fail to be Archimedean, which can be resolved by
Archimedeanization process. We define
$$
C_n^{\max} = \{  X \in M_n(\cl S \otimes \cl T): X + \epsilon (1\otimes 1)_n \in D_n^{\max} \mbox{ for all } \epsilon >0    \}.
$$
The collection $\{C_n^{\max} \}$ forms a strict, compatible matrix ordering on $\cl S\otimes \cl T$
for which $1\otimes 1$ is an Archimedean matrix order unit. We let $\cl S\otimes_{\max} \cl T$
denote the resulting tensor product. max is functorial and projective \cite{KPTT Tensor}, \cite{Han}.
By the representation of the maximal tensor product we mean the canonical identification
CP$  (\cl S \otimes_{\max} \cl T, \bb C ) \cong$CP$(\cl S, \cl T^*)$. We will need a more general form of this representation. For operator systems $\cl S$ and $\cl T$, $CB(\cl S, \cl T)$ has a matricial order structure: $(\varphi_{ij}) \in M_n(CB(\cl S, \cl T))$ is called \textit{positive} if
$$
\cl S \ni s \; \mapsto \; (\varphi_{ij}(s) ) \in M_n(\cl T) \mbox{ is a cp map}. 
$$
For a linear map $\varphi: \cl S \otimes \cl T \rightarrow \cl R$ we define $\gamma_\varphi : \cl S \rightarrow L(\cl T, \cl R)$ by $\gamma_\varphi(s) (t) = \varphi(s\otimes t)$ where $L(\cl T, \cl R)$ denotes the linear maps from $\cl T$ to $\cl R$.

\begin{theorem}\label{rep of max}
A linear map $\varphi: \cl S \otimes_{\max} \cl T \rightarrow \cl R$ is completely positive if and only if the associated map $\gamma_\varphi: \cl S \rightarrow CB(\cl T, \cl R)$ is completely positive.
\end{theorem}
\begin{proof} $(\Rightarrow)$ Let $(s_{ij}) \in M_n(\cl S)$ be positive. We wish to show that $(\gamma_\varphi(s_{ij}))$ is a completely positive map, that is, $t \mapsto (\gamma_\varphi(s_{ij})(t)) = (\varphi(s_{ij} \otimes t))  $ is a completely positive map. Let $(t_{pq}) \in M_k(\cl T)^+$ be given. We must show that $((\varphi(s_{ij}\otimes t_{pq}))_{i,j=1}^n)_{p,q=1}^k$ is a positive element of $M_{nk}(\cl R)$. This follows from the definition of maximal tensor product, in fact, $(s_{ij}) \otimes (t_{pq})$ is a positive element of $M_{nk}(\cl S \otimes_{\max} \cl T)$, hence, its image under the ${nk}^{th}$-amplification of $\varphi$ is positive.

The proof of the reverse direction is similar, so we leave it the the reader.
\end{proof}

\subsection{Commuting tensor product} We construct the commuting tensor product $\cl S \otimes_{\rm c} \cl T$ of two operator systems $\cl S$ and $\cl T$ via ucp maps with commuting ranges, that is, the collection of matricial positive cones are defined by
$$
C^{\rm com}_n = \{ X \in M_n(\cl S \otimes \cl T): \mbox{ for all Hilbert spaces } \cl H \mbox{ and for all ucp maps }
$$
$$
\hspace{3cm} \varphi: \cl S \rightarrow B(\cl H), \; \psi: \cl T \rightarrow B(\cl H)  \mbox{ with commuting ranges we have }
$$
$$
(\varphi \otimes \psi)^n (X) \geq 0 \}.
$$
We refer \cite{KPTT Tensor} for basic properties of this tensor product and recall that if one of the tensorant has a structure of a C*-algebra then c and max coincide, that is, for an operator system $\cl S$ and a C*-algebra $\cl A$ we have a canonical complete order isomorphism $\cl S \otimes_{\rm c} \cl A = \cl S \otimes_{\max} \cl A$.

\subsection{Pullback and Pushout of operator systems}
Recall that the pushout, or coproducts, of operator systems are introduced independently in \cite{kerrli} and \cite{fritz}: for a collection of operator systems $\{S_i\}_{i\in I}$ we defined the pushout as an operator system $\sqcup_i S_i$ along with the unital complete order embeddings $j_i: \cl S_i \hookrightarrow \sqcup_i \cl S_i$, $i \in I$, such that the following holds: for every operator system $\cl T$ and for every ucp maps $\varphi_i : \cl S_i \rightarrow \cl T$, $i\in I$, there exists unique ucp map $\varphi: \sqcup_i \cl S_i \rightarrow \cl T $ such that $\varphi_i = \varphi \circ j_i$ for all $i$.

We defined pullback of operator systems $\{(S_i,w_i)\}$ at the introduction. If $\cl S^*$ has a structure of an operator system (for example if dim$(\cl S) < \infty$) then we shall consider $\cl S^*$ a probability space with the faithful state $\hat e$ given by $\hat e (f) = f(e)$.

\begin{theorem}\label{duality} Let $\{\cl S_i\}_{i\in I}$ be a collection of operator systems such that for each $i\in I$, $\cl S_i^*$ is an operator system. Then we have a canonical complete order inclusion
$$
\sqcup_i \cl S_i \hookrightarrow \left( \sqcap_i \cl S_i^*   \right)^* 
$$
where the pullback $\sqcap_i \cl S_i^* $ is performed via states $\hat e_i$. If $I$ is a finite set and $dim(\cl S_i)<\infty$ for all $i$ then $\sqcup_i \cl S_i \cong \left( \sqcap_i \cl S_i^*   \right)^* $ canonically. 
\end{theorem}

\begin{proof} Let $Q_i : \sqcap_i \cl S_i^* \rightarrow \cl S_i^*$ be the canonical quotient map, and $Q_i^\dag : \cl S_i^{**} \hookrightarrow  (\sqcap_i \cl S_i^*)^* $ be its adjoint, which is a complete order embedding. We let $Q_{i,0}^\dag$ denote the restriction of $Q_i^\dag$ on $\cl S_i$. Note that $ Q_{i,0}^\dag  $ is ``unital" in the sense that $Q_{i,0}^\dag(e_i)$ is the state $\hat e$ of $\sqcap_i \cl S_i^*$ (formed via amalgamation of the states $\{\hat e_i\}_i$). Consider the Effros system
$$
[\hat e] = \mbox{span}\{ g: \mbox{ $g$ is a positive linear functional with }  g\leq \hat e   \} \subseteq (\sqcap_i \cl S_i^*)^* .
$$
Clearly the image of $Q_{i,0}^\dag$ lies in $[\hat e]$ for all $i$. The universal property of the pushout (or coproducts) ensure that we have a ucp map
$$
\gamma: \sqcup_i \cl S_i \rightarrow [\hat e] \subseteq (\sqcap_i \cl S_i^*)^*.
$$
In order to prove that $\gamma$ is a complete order embedding we will prove that its image in $ (\sqcap_i \cl S_i^*)^*$ has the same universal property that pushout has. Let $\cl T$ be a finite dimensional operator system and, for each $i$, $\varphi_i : \cl S_i \rightarrow \cl T$ be ucp maps. The adjoint map $\varphi_i^\dag: \cl T^* \rightarrow \cl S_i^*$ is a completely positive map such that for all $i$ and $j$ we have $\hat e_i \circ \varphi_i^\dag = \hat e_j \circ \varphi_j^\dag  $. The universal property of the pullback ensures that we have a completely positive map $\varphi: \cl T^* \rightarrow \sqcap_i \cl S_i^*$. If we consider the adjoint $ \varphi^\dag : (\sqcap_i \cl S_i^*)^* \rightarrow \cl T^{**} \cong \cl T $ then we obtained the desired map once we restrict $\varphi^\dag$ on the image of $\gamma$. Since the universal property of pushout can be solely verified by finite dimensional objects we obtain the proof of the first part. If $I$ is a finite set and dim$(\cl S_i)< \infty$ for all $i\in I$ then a dimension count argument proves the last argument in the theorem.
\end{proof}

\subsection{Representation of Pullback and Pushout}

We assume familiarity with universal and enveloping C*-algebra of operator systems, denoted by C*$_u(\cl S)$ and C*$_e(\cl S)$, resp., for an operator system $\cl S$ (the reader may refer to \cite{KW} and \cite{PaulsenBook}). We let $*$ denote the full free product of C*-algebras amalgamated over the unit. For a family of operator systems $\{\cl S_i\}$ we have a canonical C*-algebra isomorphism of 
$$
C^*_u(\sqcup_{i} \cl S_i) \cong *_{i} C_u^*(\cl S_i).
$$
The reader may easily verify this by using the universal property of the pushout and unital free product. Here we wish to prove that a similar inclusion also holds for the enveloping C*-algebras. As the result heavily depends on a result of F$.$ Boca \cite{Boca91}, we repeat his argument: let $\cl C$ and the family $\{\cl A_i\}$ be unital C*-algebras with a common unital C*-subalgebra $\cl B$ and  let, for each $i$,  $E_i : \cl A_i \rightarrow \cl B$ be a positive idempotent. We shall write $\cl A_i = \cl B \oplus \cl A_i^0$, where $\cl A_i^0 =$ker$(E_i)$ and $\oplus$ is the vector space direct sum. Let  $\varphi_i: \cl A_i \rightarrow \cl C $ be completely positive $\cl B$-module map for each $i$. Then there exists a completely positive $\cl B$-module map $\varphi: *_{\cl B} \cl A_i \rightarrow \cl C$ such that $\varphi(a_{i_1} a_{i_2} ...a_{i_n}) = \varphi_{i_1}(a_{i_1}) \varphi_{i_2}(a_{i_2}) \cdots \varphi_{i_n}(a_{i_n}) $ with $a_{i_j} \in \cl A_{i_j}^0$, $j = 1,2,...,n$, $n\in N$, $i_1 \neq i_2 \neq \cdots \neq i_n$.

\begin{proposition} \label{prop enveloping} For a family of operator systems $\{\cl S_i\}_i$ we have a canonical complete order inclusion
$$
\sqcup_i \cl S_i \subseteq *_i C^*_e(\cl S_i).
$$
In particular, for a family of C*-algebras $\{\cl A_i\}$ we have a canonical inclusion $
\sqcup_i \cl A_i \subseteq *_i \cl A_i.
$
\end{proposition}

\begin{proof} Let $j_i: \cl S_i \rightarrow *_i C^*_e(\cl S_i) $ be the canonical inclusion for each $i$ and let $j: \sqcup_i \cl S_i \rightarrow *_i C^*_e(\cl S_i) $ be the ucp map that appears in the definition of pushout. We wish to prove that $j$ is a complete order embedding. Let $\varphi: \sqcup_i \cl S_i \rightarrow B(\cl H)$ be a ucp map. It is sufficient to show that $\varphi$ extends to a ucp map on $ *_i C^*_e(\cl S_i)$. Let $\varphi_i$ be the restriction of $\varphi$ on $\cl S_i$ for each $i$. By Arveson's extension theorem \cite{Arveson}, let $\tilde \varphi_i$ be a ucp extension of $\varphi_i$ on $C_e^*(\cl S_i)$. We can fix a state $f_i$ on $C_e^*(\cl S_i)$, for each $i$, which plays the role of self-adjoint idempotent as in Boca's result. Now by Boca's extension theorem there exists a ucp map $\tilde \varphi : *_i C^*_e(\cl S_i) \rightarrow B(\cl H)$. It is clear that $\tilde \varphi$ is a ucp extension of $\varphi$. This proves our claim.
\end{proof}

We recall couple of facts from \cite{FKPT-discrete}. An operator subsystem $\cl S \subseteq \cl A$ of a C*-algebra $\cl A$ is called \textit{essential} if for C*-algebra $\cl B$ we have a canonical complete order embedding
$$
\cl S \otimes_{\max} \cl B \subseteq \cl A \otimes_{\max} \cl B.
$$
Every operator system is essential in its universal C*-algebra. This notion arises naturally in the study of essential tensor product. We leave the proof of the following fact to the reader:
\begin{proposition} An operator subsystem $\cl S$ of a C*-algebra $\cl A$ is essential if and only if for every operator system $\cl T$ we have a canonical inclusion
$
\cl S \otimes_{\rm c} \cl T \subseteq \cl A \otimes_{\max} \cl T.
$
\end{proposition}

Following \cite{kavruk2011} we say that an operator subsystem $\cl S \subseteq \cl A$ of a C*-algebra $\cl A$ is \textit{unitary} if C$^*\{ \cl S \cap \cl U(\cl A)  \} = \cl A$, that is, the unitaries of $\cl A$ that belongs to $\cl S$ topologically generates $\cl A$ as a C*-algebra.

\begin{theorem}\label{essentialunitary} For a family of unital C*-algebras $\{\cl A_i\}$ the canonical inclusion  $\sqcup_i \cl A_i \subseteq *_i \cl A_i$
 constitutes an essential, unitary operator subsystem. 
 \end{theorem}

 \begin{proof} Let $\cl B$ be a unital C*-algebra. Note that we have a canonical ucp map $\left( \sqcup_i \cl A_i  \right) \otimes_{\max} \cl B \rightarrow \left(*_i \cl A_i  \right) \otimes_{\max} \cl B$. This inclusion is an order embedding if and only if every state $f$ on  $\left( \sqcup_i \cl A_i  \right) \otimes_{\max} \cl B$ extends to a state $\tilde f$ on $\left( *_i \cl A_i  \right) \otimes_{\max} \cl B$. By using the representation of the maximal tensor product this is same as saying that every cp map $\varphi : \sqcup_i \cl A_i \rightarrow \cl B^*$ extends to a cp map $\tilde \varphi *i \cl A_i \rightarrow \cl B^*$. (Note that the topological dual $\cl B^*$ is viewed as a matricially ordered space.) By rescaling, if necessary, we may suppose that $\varphi(1)$ is a state in $B^*$, say $w$. Let 
 $$
 [w] = {\rm span}\{ g: 0 \leq g \leq w   \} \subset \cl B^* 
 $$
 be the corresponding Effros system. Note that $[w]$ has a structure of a von Neumann algebra, in fact $[w]$ can be identified with $\pi_w(\cl B)'$ where $\pi_w$ is the unital $*$-homomorphism that appears in the Stinespring representation of $w$. The map $\varphi : \sqcup_i \cl A_i \rightarrow [w]$ may be viewed as a collection of ucp maps $\varphi_i : \cl A_i \rightarrow [w]$. Now, by fixing states $f_i$ on each $\cl A_i$ (which plays the role of self-adjoint idempotens), Boca's result ensures that exists a ucp map $\tilde \varphi : *_i \cl A_i \rightarrow [w]$ such $\tilde \varphi|_{\cl A_i} = \varphi_i$. It is elementary to see that $\tilde \varphi$ is an extension of $\varphi$. This proves that the embedding of $\left( \sqcup_i \cl A_i  \right) \otimes_{\max} \cl B $ in $\left(*_i \cl A_i  \right) \otimes_{\max} \cl B$ is positive. If we replace $\cl B$ by $M_n(\cl B)$ then it follows that we have a complete order embedding. Since the inclusion of $\cl A_i$ in $*_i \cl A_i$ is a C*-algebra morphism  $\sqcup_i \cl A_i \subseteq *_i \cl A$ is a unitary operator subsystem.
\end{proof}

\begin{theorem}[Farenick, Paulsen] $\ell^\infty_n$ is a self-dual operator system, that is, $(\ell^\infty_n)^* \cong \ell^\infty_n$ is unitally completely order isomorphically.
\end{theorem}
So roughly speaking, considering $\ell^\infty_n$ as an operator system rather than an operator space, the classical dual $\ell_n^1$ of functionals is again an operator system which is completely order isomorphic to its predual.

\begin{theorem} We have a complete order embedding of
$$
\left( \ell^\infty_n \sqcap \ell^\infty_k \right)^* \cong \ell^\infty_n \sqcup \ell^\infty_k  \hookrightarrow {\ell^\infty_n}  *  {\ell^\infty_k} \cong C^*(\bb Z_n ) *  C^*(\bb Z_k ) \cong C^*(\bb Z_n * \bb Z_k).
$$
Moreover, the inclusion of $\left( \ell^\infty_n \sqcap \ell^\infty_k \right)^*$ in $C^*(\bb Z_n * \bb Z_k)$ constitutes a unitary operator subsytem.
\end{theorem}
\begin{proof} The first isomorphism follows from Theorem \ref{duality} and the above mentioned result of Farenick and Paulsen. Second inclusion is direct consequence of Proposition \ref{prop enveloping}. It is also well-known that the full group C*-algebra of $\bb Z_n$ is $\ell_n^\infty$. The final identification is a property of group C*-algebras and the reader may refer to \cite[page 149]{pisier_intr}. Finally note that the inclusion $ \ell^\infty_n \sqcup \ell^\infty_k  \hookrightarrow {\ell^\infty_n}  *  {\ell^\infty_k}$ constitute a unitary operator subsystem.
\end{proof}

The following result when $n_i = 2$ was proven in \cite{farenick--paulsen2011}.
 
\begin{corollary} $\sqcup_i \ell^\infty_{n_i} \subseteq *_i \ell_{n_i}^\infty \cong *_i C^*(\bb Z_{n_i}) \cong C^*(*_i \bb Z_{n_i})$ is an essential operator subsystem.
\end{corollary}

The following result indicates that the maxc tensor product of \cite{JNPPSW} actually coincides with the commuting tensor product.

\begin{corollary} We have a canonical complete order embedding
$$
\sqcup_i \ell^\infty_{n_i} \otimes_{\rm c} \sqcup_i \ell^\infty_{n_i} \subseteq *_i \ell_{n_i}^\infty  \otimes_{\max} *_i \ell_{n_i}^\infty 
$$
\end{corollary}

\begin{proof} As $\sqcup_i \ell^\infty_{n_i} \subseteq *_i \ell_{n_i}^\infty$ is an essential operator subsystem we have the following inclusions:
$$
\sqcup_i \ell^\infty_{n_i} \otimes_{\rm c} \sqcup_i \ell^\infty_{n_i} \subseteq 
*_i \ell_{n_i}^\infty  \otimes_{\max}  \sqcup_i \ell^\infty_{n_i} 
\subseteq 
*_i \ell_{n_i}^\infty  \otimes_{\max} *_i \ell_{n_i}^\infty .
$$
\end{proof}

\begin{theorem}\label{S vs A} Let $\cl S \subseteq \cl A$ be an essential, unitary operator subsystem of a C*-algebra $\cl A$. Then for every C*-algebras $\cl B_1 \subseteq \cl B_2$ we have
$$
\cl S \otimes_{\max} \cl B_1 \subseteq \cl S \otimes_{\max} \cl B_2 \;\; \Longleftrightarrow \;\;  \cl A \otimes_{\max} \cl B_1 \subseteq \cl A \otimes_{\max} \cl B_2.
$$
\end{theorem}
\begin{proof}
We shall only prove the non-trivial direction $\Rightarrow$. Let $\tau$ be the C*-algebra structure on $\cl A \otimes \cl B_1$ arising from the inclusion $\cl A \otimes_{\max} \cl B_2$. Since 
$$
\cl S \otimes_{\max} \cl B_1 \subseteq \cl S \otimes_{\max} \cl B_2 \subseteq \cl A \otimes_{\max} \cl B_2 \supseteq \cl A \otimes_{\tau} \cl B_1
$$
we have a canonical complete order embedding $\cl S \otimes_{\max} \cl B_1 \subseteq \cl A \otimes_{\tau} \cl B_1$. It is not hard to see that  $\cl S \otimes_{\max} \cl B_1 $ is a unitary operator subsystem of $\cl A \otimes_{\tau} \cl B_1$. By Proposition 5.6 of \cite{kavruk2011} it follows that $\cl A \otimes_{\tau} \cl B_1$ must the enveloping C*-algebra of  $\cl S \otimes_{\max} \cl B_1 $. On the other hand the canonical inclusion $\cl S \otimes_{\max} \cl B_1 \subseteq \cl A \otimes_{\max} \cl B_1$ also forms a unitary operator system. Therefore the equivalence of $\tau$ and $\max$ on $\cl A \otimes \cl B_1$ follows from the uniqueness and rigidity of the enveloping C*-algebra. This proves our claim.
\end{proof}

\subsection{Non-commutative polygones, Namioka-Phelps test systems} The pullback and pushout of operator systems we discuss in the previous sections exhibit familiar objects: by a \textit{non-commutative polygon} with $k$-input, $n$-output, denoted by $\cl {NP}_{n,k}$, we mean pushout of $k$-copies of $\ell^\infty_n$. Likewise a \textit{non-commutative cube} with $k$-input (and binary output), denoted by $\cl {NC}_{k}$, is pushout of $k$-copies of $\ell_2^\infty$:
$$
\cl {NP}_{n,k} : = \bigsqcup_{i=1}^k \ell_n^\infty
\;\;\;\;\;\;\;\; 
\cl {NC}_{k} : = \bigsqcup_{i=1}^k \ell_2^\infty.
$$

Note that for any operator systems $\cl S$ and for any positive elements $\{s_x^a\}_{x=1}^n$ with $\sum_x s_x^a = 1$, where the upper index $a$ ranges from 1 to $k$, we have a ucp map $\varphi : \cl {NP}_{n,k} \rightarrow \cl S$ given by $e_x^a \mapsto s_x^a $ where $e_x^a$ is the $x^{th}$ canonical projection that appears in the $a^{th}$ copy of $\ell_n^\infty$.
Regarding $\cl {NC}_{k}$ we identify $e_1 = e_1^1$, $e_2 = e_1^2$,...,$e_k = e_1^k$. Therefore for any operator system $\cl S$ and $0\leq s_1,...,s_k \leq 1$ we have a ucp map from $\cl {NC}_k$ to $\cl S$ given by $e_i \mapsto s_i$, $i=1,..,k$, and $e\mapsto 1$.

\smallskip

The pullback $\cl V = \ell_2^\infty \sqcap \ell_2^\infty \cong \{(a,b,c,d): a+b = c+d\} \subset \ell_4^\infty 
$ is referred as the standard Namioka-Phelps test system \cite{NP}. A generalized version can be given by
$$
\cl V_{n,k} := \rotatebox[origin=c]{180}{$\displaystyle \bigsqcup$}_{i=1}^k \ell_n^\infty.
$$
By the Theorem \ref{duality} and the above mentioned result of Paulsen and Farenick we have complete order isomorphism
$$
\cl {NP}_{n,k}^* \cong \cl V_{n,k}.
$$
These objects naturally appears in the study of nuclearity related properties in positivity and non-commutative order theory, joint quantum probability distributions in Quantum Information Theory \cite{NP}, \cite{Effros}, \cite{JNPPSW}, \cite{FKPT-discrete}  etc. 

\subsection{W.R.I$.$ and nuclearity} Given a pair of C*-algebras $\cl A \subseteq \cl B$, we say that $\cl A$ is \textit{relatively weakly injective} in $\cl B$ (or $\cl A$ has the weak relative injectivity in $\cl B$), $\cl A$ is w.r.i$.$ in $\cl B$ in short, if the canonical inclusion of $\cl A$ into bidual von Neuman algebra $\cl A^{**}$ extends to ucp map on $\cl B$:
$$
\xymatrix{
\cl A \ar@{^{(}->}[rr]^{\hat c} \ar@{_{(}->}[d]_i & & \cl A^{**} \\
\cl B \ar[rru]_{\varphi} & &
}.
$$

We begin with the following characterization:

\begin{theorem}[Kirchberg, \cite{Kirchberg94}] The following statements are equivalent for $\cl A \subseteq B$:
\begin{enumerate}
\item $\cl A$ is w.r.i$.$ in $\cl B$;
\item for every C*-algebra $\cl C$ we have a canonical inclusion $\cl A \otimes_{\max} \cl C \subseteq \cl B \otimes_{\max} \cl C $;
\item we have canonical C*-algebra inclusion $\cl A \otimes_{\max} C^*(\bb F_\infty) \subseteq \cl B \otimes_{\max} C^*(\bb F_\infty)$.
\end{enumerate}
\end{theorem}

\begin{remark} In the above theorem $C^*(\bb F_\infty)$ can be replaced by 
\begin{enumerate}
\item $C^*(\bb F_n)$ where $n\geq 2$;
\item $C^*(\bb Z_n * \bb Z_k) \cong \ell_n^\infty * \ell_k^\infty$, for any $n,k \geq 2$ with $n+k\geq 5$,
\item  $C^*(*_{i\in I} \bb Z_{n_i}  ) \cong *_{i\in I} \ell^\infty_{n_i}$ with $|I|\geq 3$ and $n_i \geq 2$ for at least three distinct $i$.
\end{enumerate}
\end{remark}

\begin{proof} We let $G$ be the set of the discrete groups $\bb F_n$, $n\geq 2$, $\bb Z_n * \bb Z_k$ with $n,k \geq 2$ provided $n+k\geq 5$ and $*_{i\in I}\bb Z_{n_i}$ provided $n_i\geq 2$ for at least three distinct values of $i$. It is well known that for any group $\bb G \in G$ we have a group inclusion $\bb F_\infty \hookrightarrow \bb G$. (For example the reader may refer to \cite{Harpe} for an excellent source.) This implies, by Section 8 of \cite{pisier_intr}, that we have a canonical C*-algebra inclusion $i:C^*(\bb F_\infty) \hookrightarrow C^*(\bb G)$ with a ucp inverse, say $\varphi$. Now given a pair of C*-algebras $\cl A \subseteq \cl B$ if we have a C*-algebra inclusion  $\cl A \otimes_{\max} C^*(\bb G) \subseteq \cl B \otimes_{\max} C^*(\bb G)$ then it follows that  $\cl A \otimes_{\max} C^*(\bb F_\infty) \subseteq \cl B \otimes_{\max} C^*(\bb F_\infty)$. In fact the following composition
$$
\cl A \otimes_{\max} C^*(\bb F_\infty) \xrightarrow{id \otimes i}  \cl A \otimes_{\max} C^*(\bb G) \xrightarrow{id \otimes \varphi}  \cl A \otimes_{\max} C^*(\bb F_\infty),
$$
which is identity, implies that the first map is a C*-algebra inclusion (recall that if the composition of two ucp maps is a complete order embedding then this is the case for the first map). Likewise we have a C*-algebra inclusion $\cl B \otimes_{\max} C^*(\bb F_\infty) \subseteq  \cl B \otimes_{\max} C^*(\bb G) $ given by $id\otimes i$. Since $\cl A \otimes_{\max} C^*(\bb G) \subseteq \cl B \otimes_{\max} C^*(\bb G)$ it follows that $\cl A \otimes_{\max} C^*(\bb F_\infty) \subseteq \cl B \otimes_{\max} C^*(\bb F_\infty)$.
\end{proof}

\begin{theorem}\label{main WRI} Let $\mathfrak S$ be the set of all the operator systems of the form $\ell_n^\infty \sqcup \ell^\infty_k$ with $n,k\geq 2$ provided $n+k\geq 5$ and $\sqcup_{i\in I} \ell^\infty_{n_i}$ with $|I|\geq 3$ and $n_i\geq 2$ for at least three distinct $i$. Let $\cl A \subseteq \cl B$ be given. The following are equivalent:
\begin{enumerate}
\item $\cl A$ is w.r.i$.$ in $\cl B$;
\item for all operator system  $\cl S$ we have a canonical inclusion $\cl A \otimes_{\max} \cl S \subseteq \cl B \otimes_{\max} \cl S $;
\item there exists an operator system $\cl S_0 \in \mathfrak S$ such that we have complete order inclusion $$
\cl A \otimes_{\max} \cl S_0 \subseteq \cl B \otimes_{\max} \cl S_0.
$$
\end{enumerate}
\end{theorem}
\begin{proof}  (1) $\Rightarrow$ (2): Let the operator system $\cl S$ be given. We use \cite[Lemma 6.5]{KPTT Nuclearity}: $\cl A \otimes_{\max} \cl S \subseteq \cl A^{**} \otimes_{\max} \cl S$ completely order isomorphically, that is, $\hat c \otimes id$ is a unital complete order embedding from $\cl A \otimes_{\max} \cl S$ to $\cl A^{**} \otimes_{\max} \cl S$ where $\hat c$ denotes the canonical embedding of $\cl A$ in to $\cl A^{**}$ .  Letting $\varphi: \cl B \rightarrow \cl A^{**}$ be the ucp map extending $\hat c$, the following composition of ucp maps
$$
\cl A \otimes_{\max} \cl S \xrightarrow{i \otimes id} \cl B \otimes_{\max} \cl S  \xrightarrow{\varphi \otimes id} \cl A^{**} \otimes_{\max} \cl S
$$
is a complete order embedding. So this is the case for the first map, which proves our claim.

\smallskip

(2) $\Rightarrow $ (3) is clear.

\smallskip

(3) $\Rightarrow$ (1): Let $\cl S_0 \in \mathfrak S$ be given. We will suppose that $\cl S_0 = \ell_n^\infty \sqcup \ell^\infty_k$ with $n,k\geq 2$ provided $n+k\geq 5$. In fact the other case can be treated identically. We know by Theorem \ref{essentialunitary} that $\cl S_0 \subseteq \ell_n^\infty * \ell^\infty_k $ forms an essential, unitary operator subsystem. Since $\cl A \otimes_{\max} \cl S_0 \subseteq \cl B \otimes_{\max} \cl S_0$, by Theorem \ref{S vs A}, we have a C*-algebra inclusion
$$
\cl A \otimes_{\max} \left( \ell_n^\infty * \ell^\infty_k \right) \subseteq \cl B \otimes_{\max}  \left( \ell_n^\infty * \ell^\infty_k \right).
$$
Now by the above theorem of Kirchberg (and the subsequent remark) $\cl A$ is w.r.i$.$ in $\cl B$.
\end{proof}

Recall that an operator system $\cl S$ is said to be \textit{C*-nuclear} if for every C*-algebra $\cl A$ we have a canonical complete order isomorphism
$$
\cl S \otimes_{\min} \cl A = \cl S \otimes_{\max}  \cl A.
$$
If $\cl S$ is C*-nuclear then for every C*-algebras $\cl A \subseteq \cl B$ we have a canonical complete order inclusion $\cl A \otimes_{\max} \cl S \subseteq \cl B \otimes_{\max} \cl S $, in fact one can replace max by min which is injective. The following is logically equivalent to amenability of  $\bb Z_2 * \bb Z_2$:

\begin{proposition}[Cor$.$ 6.13 of \cite{kavruk2012}]\label{NP_2} $\cl {NP}_2 = \ell_2^\infty\sqcup \ell^\infty_2$ is C*-nuclear.
\end{proposition}

We close this section with a brief summary on the weak expectation property (WEP) which is introduced by C. Lance \cite{Lance}, \cite{Lance2}.  A unital C*-algebra is said to have the \textit{weak expectation property} if the canonical inclusion of $\cl A$ into bidual von Neumann algebra $\cl A^{**}$ decomposes via ucp maps through an injective C*-algebra. In particular if we represent $\cl A$ into a $B(\cl H)$, then $\cl A$ has WEP if and only if the canonical inclusion of $\cl A$ into $\cl A^{**}$ extends to a ucp map on $B(\cl H)$. Therefore a unital C*-algebra $\cl A$ has WEP if and only if it is w.r.i$.$ in $B(\cl H)$ for some representation $\cl A \subseteq B(\cl H)$.   An operator system $\cl S$ is said to have the \textit{local lifting property} if for every C*-algebra $\cl A$ and ideal $I$ in $\cl A$ the following holds: for every ucp map $\varphi: \cl S \rightarrow \cl A/I $ and for every finite dimensional subspace $\cl S_0 \subseteq \cl S$ $\varphi|_{\cl S_0}$ has a ucp (equivalently cp) lift on $\cl A$. If the $\cl S$ is finite dimensional or a global lifting exists we simply use the term \textit{lifting property}. The reader may refer to \cite{KPTT Nuclearity} for more detail on the topic. Extending a well-known result of Kirchberg \cite{Kirchberg94}, an operator system $\cl S$ has the local lifting property if and only if we have
$$
\cl S \otimes_{\min} B(\cl H) = \cl S \otimes_{\max} B(\cl H)
$$
canonically, where $\cl H$ is a separable Hilbert space. Either by using the definition with standard Paulsen trick \cite{PaulsenBook} or the tensor characterization one can easily shows that $\cl S$ has the local lifting property if and only if $M_n(\cl S)$ has the same property for every $n$. We in particular need the fact that $C^*(\bb F_\infty)$ has the local lifting property \cite{Kirchberg94}. Also the reader may verify that if $\{\cl S_i\}$ is family of operator systems with the lifting property then $\sqcup_i \cl S_i$ has the lifting property \cite{kavruk2011}.

\section{Proof of the Main Results}

In this section we will prove the main results we listed in the preface and provide the promised example regarding the behaviour of Arveson-Riesz property in the general operator system category.
\begin{example}\label{1. example}
In Theorem \ref{22 RA Prop} we have seen that (2,2)-Riesz Arveson Property holds in the category of C*-algebras. Here we will show that it fails in general operator systems. Consider the Namioka-Phelps test system
$$
\cl V = \{(a_1,a_2,a_3,a_4): a_1+a_2 = a_3+a_4\} \subset \ell^\infty_4.
$$
Consider the states $\varphi_i: \cl V \rightarrow \bb C$ defined by $\varphi_i(a_1,a_2,a_3,a_4) = a_j$ for $i=1,2,3,4$. Clearly $\varphi_1 + \varphi_2 = \varphi_3 + \varphi_4$. We claim that there are no positive extensions $\tilde \varphi_i$ of $\varphi_i$ on $\ell_4^\infty$, $i=1,2,3,4$, with the property that $\tilde \varphi_1 + \tilde \varphi_2 = \tilde \varphi_3 + \tilde \varphi_4$. Assume for a contradiction that such maps exist. Then we have a well-defined ucp map
$$
\gamma: \ell_4^\infty \rightarrow \ell_4^\infty \mbox{ given by } \gamma(x) = (\tilde \varphi_1(x) , \tilde \varphi_2(x) , \tilde \varphi_3(x) , \tilde \varphi_4(x)).
$$
We first note that the image of $\gamma$ is contained in $\cl V$ and secondly the restriction of $\gamma$ on $\cl V$ is the identity on $\cl V$. We leave the elementary verification of these facts to the reader. Therefore we may regard $\gamma$ as a ucp extension of the identity on $\cl V$
$$
\xymatrix{
\cl V \ar@{^{(}->}[rr]^{id} \ar@{_{(}->}[d]_i & & \cl V \\
\ell_\infty^4 \ar[rru]_{\gamma} & &
}
$$
But then a standard Choi-Effros trick forces $\cl V$ to have a structure of a C*-algebra, which yields a contradiction. In fact there are many ways to see that $\cl V$ cannot be operator system isomorphic to $\ell^\infty_3$ (perhaps the most versatile method is that $\cl V$, as a Kadison space, is a test object to verify nuclearity so cannot be order isomorphic to $\ell_3^\infty$.) 
\end{example}

\begin{proof}[proof of Theorem 0.1] (1)  Let $[\varphi] \subseteq CB(\cl A, \cl C)$ be the Effros system associated with $\varphi$. Consider the map  $\gamma : \cl {NC}_2 \rightarrow [\varphi] \subseteq CB(\cl A, \cl C)$ defined by $1\mapsto \varphi$, $e_1 \mapsto \varphi_1$, $e_2 \mapsto \varphi_2$. The universal property of the non-commutative cube $\cl {NC}_2$ ensures that $\gamma$ is a ucp map. By the representation of the maximal tensor product, $\gamma$ corresponds to a completely positive map $\Gamma: \cl {NC}_2 \otimes_{\max} \cl A \rightarrow \cl C$. Since $\cl {NC}_2$ is C*-nuclear operator system (Proposition \ref{NP_2}) we have an obvious inclusion $\cl {NC}_2 \otimes_{\max} \cl A \subseteq \cl {NC}_2 \otimes_{\max} \cl B$. As $\cl C$ is an injective C*-algebra, $\Gamma$ extends to a completely positive map $\tilde \Gamma: \cl {NC}_2 \otimes_{\max} \cl B \rightarrow \cl C$. Let $\tilde \gamma$ be the corresponding completely positive map from  $\cl {NC}_2$ to $CB(\cl B, \cl C)$. It is elementary to see that $\tilde \gamma(\cdot)|_{\cl A} = \gamma$. Now setting $\tilde \varphi =: \tilde \gamma(1)$, $\tilde \varphi_1 =: \tilde \gamma(e_1)$, $\tilde \varphi_2 =: \tilde \gamma(e_2)$, we obtain the desired extensions.

$ $

(2) Let $\ell^\infty_2 \sqcap \ell_2^\infty =  \cl V = \{(a,b,c,d): a+b = c+ d\} \subset \ell_4^\infty$ be the standard Namioka-Phelp test system. Note that $\cl C \otimes_{\min} \cl V$ can be identified with $\{(c_1,c_2, c_3,c_4: c_1 + c_2 = c_3 + c_4)\} \subset \cl C \otimes \ell_4^\infty$ $\cong$ $ \oplus_{i=1}^4 \cl C $. Consider the map $\gamma: \cl A \rightarrow \cl C \otimes_{\min} \cl V$ given by $a \mapsto (\varphi_1(a), \varphi_2 (a), \varphi_3(a), \varphi_4(a))$, which is a completely positive map. By the representation of the minimal tensor product we have the identification of $CB(\cl V^*, \cl C) \cong \cl C \otimes_{\min} \cl V$. Moreover, by the representation of the maximal tensor product, $\gamma: \cl A \rightarrow CB(\cl V^*, \cl C)$ corresponds a completely positive map $\Gamma: \cl A \otimes_{\max} \cl V^* \rightarrow \cl C$. By using the fact that $\cl V^* \cong \cl {NC}_2$ and $\cl {NC}_2$ is C*-nuclear, we deduce that $\cl A \otimes_{\max} \cl V^* \subseteq \cl B \otimes_{\max} \cl V^*$. As $\cl C$ is an injective C*-algebra, there exists a completely positive map $\tilde \Gamma: \cl B \otimes_{\max} \cl V^* \rightarrow \cl C$ that extends $\Gamma$. Now following the procedures backward we obtain a completely positive map $\tilde \gamma: \cl B  \rightarrow \cl C \otimes_{\min} \cl V $ that extends $\gamma$. Now the components of this map satisfy the desired properties.
\end{proof}

\textit{Remark.} It is not hard to show that the statements of Theorem 0.1 are equivalent. However we prove these statements separately as the proofs rely on duality corresponding of pullbacks and pushout of operator systems.

\begin{proof}[proof of Theorem 0.2] (1) $\Rightarrow $ (2): Let cp maps $\varphi_1, \varphi_2, \varphi_3, \varphi : \cl A \rightarrow M_m$ with $\varphi_1,\varphi_2, \varphi_3 \leq \varphi$ be given. Recall, by Theorem \ref{main WRI}, that $\cl A$ being w.r.i$.$ in $\cl B$ is equivalent to the statement that $\cl A \otimes_{\max} \cl {NC}_3 \subseteq \cl B \otimes_{\max} \cl {NC}_3$ completely order isomorphically, where $\cl {NC}_3 = \ell^\infty_2 \sqcup \ell^\infty_2 \sqcup \ell^\infty_2$ denotes the non-commutative cube. Consider the Effros system $$[\varphi] = {\rm span} \{ \phi: \phi \mbox{ is a cp map with } \phi \leq \varphi \} \subseteq CB(\cl A, M_m).$$ Note that $\gamma: \cl {NC}_3 \rightarrow [\varphi] \subseteq CB(\cl A, M_m)$ given by $e_1 \mapsto \varphi_1$, $e_2 \mapsto \varphi_2$, $e_2 \mapsto \varphi_3$, $e \mapsto \varphi$ declares a ucp map by the universal property of the non-commutative cube. By Theorem \ref{rep of max}, $\gamma$ corresponds to a cp map $\Gamma: \cl A \otimes_{\max} \cl {NC}_3 \rightarrow M_m$. Since $M_m$ is injective, $\Gamma$ extends to a cp map $\tilde \Gamma: \cl B \otimes_{\max} \cl {NC}_3 \rightarrow M_m$. Now if we consider the associated map $\tilde \gamma : \cl {NC}_3 \rightarrow CB(\cl B, M_m)$ we obtain that the maps $\tilde \gamma(e_1)$, $\tilde \gamma(e_2)$, $\tilde \gamma(e_3)$ and $\tilde \gamma(e)$ are the desired extensions of $\varphi_1$, $\varphi_2$, $\varphi_3$ and $\varphi$, respectively.

\smallskip

 (2) $\Rightarrow $ (1): Consider the restriction map $R: CB(\cl B, M_m) \rightarrow CB(\cl A, M_m) $ given by $\theta \mapsto \theta|_{\cl A}$. Note that there is a one-to-one correspondence between cp maps  $\varphi_1, \varphi_2, \varphi_3, \varphi : \cl A \rightarrow M_m$ with $\varphi_1,\varphi_2, \varphi_3 \leq \varphi$ and $ \gamma(e_1)$, $ \gamma(e_2)$, $ \gamma(e_3)$ and $ \gamma(e)$ where $\gamma$ is a cp map from $\cl {NP}_3$ to ${\rm CB}(\cl A, M_m)$. Therefore the statement in (2) is equivalent to every cp $\gamma: \cl {NP}_3 \rightarrow   CB(\cl A, M_m)$ lifts to cp map on $ CB(\cl B, M_m)$:
 $$
\xymatrix{
 \cl {NP}_3 \ar[rr]^{\gamma} \ar[rrd]_{\tilde \gamma}& & CB(\cl A, M_m)   \\
& & CB(\cl B, M_m) \ar[u]_{R}
}.
$$
By the representation of the maximal tensor product this is same as saying that for every cp map $\Gamma :  \cl A \otimes_{\max} \cl {NC}_3 \rightarrow M_m$  there exits a cp map $\tilde \Gamma :  \cl B \otimes_{\max} \cl {NC}_3 \rightarrow M_m$ such that the following composition
$$
 \cl A \otimes_{\max} \cl {NC}_3 \xrightarrow{i\otimes id}  \cl B \otimes_{\max} \cl {NC}_3 \xrightarrow{\;\;\;\;\tilde \Gamma\;\;\;\;} M_m
$$
coincides with $\Gamma$. This shows that the canonical map $\cl A \otimes_{\max} \cl {NC}_3 \xrightarrow{i\otimes id}  \cl B \otimes_{\max} \cl {NC}_3 $ is $m$-order inclusion. Since $m$ is arbitrary it follows that we have a complete order inclusion $\cl A \otimes_{\max} \cl {NC}_3 \subseteq \cl B \otimes_{\max} \cl {NC}_3$. Now by Theorem \ref{main WRI}, $\cl A$ is w.r.i$.$ in $\cl B$.

\smallskip

(2) $\Leftrightarrow$ (4): It's not hard to see that the conditions in (2) and (4) are equivalent. In fact, assuming (2), given $\varphi_i$, $i=1,...,6$, one can proceed by setting $\varphi = \varphi_1 +  \varphi_2$ which dominates $\varphi_1$, $\varphi_3$ and $\varphi_5$. Conversely, assuming $(4)$, given $\varphi_1,\varphi_2, \varphi_3, \varphi$ as in (2), one can proceed by considering the pairs of cp maps $\varphi_1$, $\varphi-\varphi_1$, $\varphi_2$, $\varphi-\varphi_2$ and $\varphi_3$, $\varphi-\varphi_3$ whose sums are equal. We leave the details to the reader.

\smallskip

(1) $\Leftrightarrow$ (3) The steps of the proof follow the same patterns as in (1) $\Leftrightarrow$ (2), in fact one can simply replace $\cl {NP}_3  = \ell_2^\infty \sqcup \ell_2^\infty \sqcup \ell_2^\infty$ by $\ell_2^\infty \sqcup \ell_3^\infty$. For example given $\varphi_i$, $i=1,...,5$ with $\varphi_1 + \varphi_2 = \varphi_3 + \varphi_4+ \varphi_5$ we can define a cp map from $\ell_2^\infty \sqcup \ell_3^\infty $ to $CB(\cl A, M_m)$ by $e_1^{1} = \varphi_1$, $e_2^{1} = \varphi_2$, $e_1^{2} = \varphi_3$, $e_2^{2} = \varphi_4$, $e_3^{2} = \varphi_5$ where $\{e_1^1,e_2^1 \}$ and $\{e_1^2,e_2^2, e_3^2 \}$ are the canonical basises for $\ell_2^\infty$ and $\ell_3^\infty$, respectively. One can proceed by using the fact that $\cl A \subseteq \cl B$ relatively weakly injective if and only if $\cl A \otimes_{\max} \left(\ell_2^\infty \sqcup \ell_3^\infty \right)\subseteq \cl B \otimes_{\max} \left(\ell_2^\infty \sqcup \ell_3^\infty\right) $ completely order isomophically (Theorem \ref{main WRI}). We leave the details to the reader.
\end{proof}

\smallskip

For a pair of C*-algebras $\cl A \subseteq B$ we say that $\cl A$ has the ($n,k$)-\textit{Riesz-Arveson property in} $\cl B$ for for every $m$, every cp maps $\varphi_i:\cl A \rightarrow M_m$, $i=1,...,n+k$, with
$$
\sum_{i=1}^n \varphi_i = \sum_{i=n+1}^{n+k} \varphi_i
$$
have cp extensions $\tilde \varphi_i:\cl B \rightarrow M_n$, $i=1,...,n+k$, resp$.$, with
$$
\sum_{i=1}^n \tilde \varphi_i = \sum_{i=n+1}^{n+k} \tilde \varphi_i.
$$
Simply following the arguments in the above proof one can show that $\cl A$ is w.r.$i$. in $\cl B$ if and only if it has $(n,k)$-Riesz-Arveson property in $\cl B$ for every $n$ and $k$, (equivalently for some $n$ and $k$ satisfying $n\geq 2$, $k\geq 2$ with $n+k\geq 5$). Here the ranges $\{M_m\}_{m=1}^\infty$ can be replaced by $B(\cl H)$ where $\cl H$ is a separable Hilbert space or arbitrary injective objects. It is also possible sharpen (2) of Theorem 0.2:

\begin{theorem} The following statements are equivalent for $\cl A \subseteq \cl B$:
\begin{enumerate}
\item $\cl A$ is w.r.i$.$ in $\cl B$;
\item for any injective object $\cl C$ and cp map $\varphi:\cl A \rightarrow \cl C$ we have a cp extension $\tilde \varphi: \cl B \rightarrow \cl C$ of $\varphi$ such a way that every cp map $\varphi_0:\cl A \rightarrow \cl C$ with $\varphi_0 \leq \varphi$ has a cp extension $\tilde \varphi_0: \cl B \rightarrow \cl C$ with $\tilde \varphi_0 \leq \tilde \varphi$. Moreover this can be achieved such a way that $\varphi_0 \mapsto \tilde \varphi_0$ is a ucp map from $[\varphi] \subseteq CB(\cl A,\cl C)$ to $[\tilde \varphi] \subseteq CB(\cl B,\cl C)$ (with a ucp inverse given by the restriction).
\end{enumerate}
\end{theorem}

\begin{proof} (2)$\Rightarrow$(1). (2) implies Theorem 0.2 (2) so $\cl A$ is w.r.i$.$ in $\cl B$.
\smallskip

(1)$\Rightarrow$(2): The evaluation map $E: \cl A\otimes_{\max} [\varphi] \rightarrow \cl C$ given by $E(a\otimes \psi) = \psi(a)$ is a cp map. As $\cl A \otimes_{\max} [\varphi] \subseteq \cl B \otimes_{\max} [\varphi]$, we have a cp extension $\tilde E : \cl B \otimes_{\max} [\varphi] \rightarrow \cl C$ of $E$. By the representation of the max $\tilde E$ corresponds to a cp map $\gamma: [\varphi] \rightarrow CB(\cl B,\cl C)$. Setting $\tilde \varphi : = \gamma(\varphi)$ the reader may verify that we have all the above mentioned properties.
\end{proof}

\smallskip

We now turn out attention to tight Riesz interpolation property. Recall that TR(2,2)-property holds for arbitrary pairs of C*-algebras $\cl A\subseteq \cl B$. It fails in general operator systems:

\begin{example} \label{2. example} Consider $\cl V = \{(a,b,c,d): a+b = c+d\} \subseteq \ell^\infty_4$, the standard Namioka-Phelps test system, and the elements
$$
x_1 = (-3,1,-1,-1),\; x_2 = (1,-3,-1,-1) ,\; x_3 = (2,2,4,0),\; x_4 = (2,2,0,4).
$$ Note that $x_1,x_2$ and $x_3,x_4$ interpolate in $\ell_4^\infty$, in fact if we set $y = (3/2, 3/2, -1/2, -1/2)$ then clearly $x_1, x_2 < y < x_3, x_4$. However there exists no $y_0 \in \cl V$ with this property. We leave the elementary verification to the reader.
\end{example}

Before we proceed we recall the quotient characterization of the pushout: for operator systems $\cl S_i$, $i=1,...,n$, we have a canonical complete order isomorphism 
$$
\left(\bigoplus_{i=1}^n \sl S_i\right)/J \cong \bigsqcup_{i=1}^n \sl S_i.
$$
where $J$ is null-subspace given by $J=$span$\{ (e,-e,0,...,0),  (e,0,-e,0,...,0),  (e,0,...,-e)\}$ (see Sec$.$ 5 and Theorem 5.1 of \cite{FKTT Char.WEP}). In particular we have a complete order isomorphism
$$
\ell^\infty_{n} \oplus \ell^\infty_{k}/ {\rm span}\{(e,-e)\} \cong \ell^\infty_{n+k}/ {\rm span}\{ (\underbrace{1,1,..,1}_{n-terms},\underbrace{-1,-1,..,-1}_{k-terms}) \} \cong \ell_n^\infty \sqcup \ell_k^\infty.
$$
Based on the projectivity of the maximal tensor product \cite{Han} we have:

\begin{lemma}\label{Lem TR char} Let $\cl S$ be an operator system and $s_1,s_2,...,s_{n+k}$ are self-adjoint elements of $\cl S$. Then the following statements are equivalent:

\smallskip

\begin{enumerate}

\item $\sum_{i=1}^{n+k} s_i \otimes \dot e_i > 0$ in $\cl S \otimes_{\max}  \ell^\infty_{n+k} / {\rm span}\{(e_n,-e_k)\}$;

\smallskip

\item there exists a self-adjoint element $s \in \cl S$ such that $-s_i < s$ for $i=1,...,n$ and $s< s_j$ for $j = n+1,...,n+k$. 

\end{enumerate}
\end{lemma}

\begin{proof} (1) $\Rightarrow $ (2). Since $\sum_{i=1}^{n+k} s_i \otimes \dot e_i > 0$ we can find $\delta>0$ such that  $\sum_{i=1}^{n+k} s_i \otimes \dot e_i  - \delta (e\otimes \dot{e})> 0$. The projectivity of the maximal tensor product ensures that we have a canonical complete order isomorphism
$$
\cl S \otimes_{\max}  \ell^\infty_{n+k} / {\rm span}\{(e_n,-e_k)\} \cong \left( \cl S \otimes_{\max} \ell_{n+k}^\infty \right) / (\cl S \otimes {\rm span}\{(e_n,-e_k)\}).
$$
Since $\sum_{i=1}^{n+k} s_i \otimes \dot e_i  - \delta e\otimes \dot e$ is a strictly positive element, it must have a positive representative in $\cl S \otimes_{\max} \ell_{n+k}^\infty$, say 
$$
\sum_{i=1}^{n+k} s_i \otimes e_i  - \delta e\otimes  e + s\otimes (e_n, -e_k).
$$
As $ \ell_{n+k}^\infty$ is a nuclear we have the identification   $\cl S \otimes_{\max} \ell_{n+k}^\infty \cong \cl S \otimes_{\min} \ell_{n+k}^\infty \cong \oplus_{i = 1}^{n+k} \cl S$. Now the element above can be written as
\begin{eqnarray}\label{poselement}{
(s_1,s_2,...,s_{n+k}) -\delta (1,1,...,1) +
(\underbrace{s,s,..,s}_{n-terms},\underbrace{-s,-s,..,-s}_{k-terms}) \} \geq 0.
}\end{eqnarray}
Therefore we have $s_i - \delta 1 + s \geq 0$ for $i = 1,...,n$ and $s_i - \delta 1 - s \geq 0$ for $i = n+1,...,n+k$. Now clearly the statement in (2) holds.

\smallskip

(2) $\Rightarrow $ (1). So we can find $\delta > 0$ such that $s_i - \delta 1 + s \geq 0$ for $i = 1,...,n$ and $s_i - \delta 1 - s \geq 0$ for $i = n+1,...,n+k$. This implies that the element that appears in (2.1) is strictly positive in $ \oplus_{i = 1}^{n+k} \cl S \cong \cl S \otimes_{\max} \ell_{n+k}^\infty$. By using the definition of the quotient (and the projectivity of the maximal tensor product) the element $ \sum_{i=1}^{n+k} s_i \otimes \dot e_i  - \delta e\otimes \dot e$ is positive in $\cl S \otimes_{\max}  \ell^\infty_{n+k} / {\rm span}\{e_n,-e_k\}$. This proves our claim.
\end{proof}

\begin{lemma}\label{TR & tensor} Let $\cl A $ be a unital C*-subalgebra of $  \cl B$. The following are equivalent:
\begin{enumerate}
\item $\cl A$ has TR(n,k)-property in $\cl B$;

\item  we have a canonical order embedding
$
\cl A \otimes_{\max} \left( \ell_n^\infty \sqcup \ell_k^\infty \right) \subseteq 
\cl B \otimes_{\max} \left( \ell_n^\infty \sqcup \ell_k^\infty \right).
$
\end{enumerate}
\end{lemma}

\begin{proof} Before we start the proof we identify $  \ell_n^\infty \sqcup \ell_k^\infty \cong  \ell_{n+k}^\infty/{\rm span}\{ (e_n, -e_k) \} $. Note that $\{\dot e_i\}_{i=1}^{n+k}$ is not basis for the quotient, which can be called a frame. However if $\sum_{i} a_i \otimes \dot e_i$ is a self-adjoint element then we may suppose that each $a_i$ is self-adjoint.

\smallskip

(1) $\Rightarrow$ (2). To show that the canonical map is an order embedding we need to prove the following: if $\sum a_i \otimes \dot e_i$ is positive in $\cl B \otimes_{\max} \left(\ell_{n+k}^\infty/{\rm span}\{ (e_n, -e_k) \} \right)$ then it must be positive in $\cl A \otimes_{\max} \left(\ell_{n+k}^\infty/{\rm span}\{ (e_n, -e_k) \} \right)$. Here first note that we may suppose each $a_i$ is self-adjoint. Also, a moment of thought shows that we may suppose that $\sum a_i \otimes \dot e_i$ is a strictly positive element. The above lemma shows that there exists $b$ in $\cl B$ such that $- a_i< b < a_j$ for $i=1,...,n$ and $j = n+1,...,n+k$. As $\cl A$ has TR$(n,k)$-property in $\cl B$, it follows that there is an element $a \in \cl A$ with the same property. But this means, by the above lemma, that  $\sum a_i \otimes \dot e_i$ is a strictly positive element of   $\cl A \otimes_{\max} \left(\ell_{n+k}^\infty/{\rm span}\{ (e_n, -e_k) \} \right)$. This proves our claim.

\smallskip

(2) $\Rightarrow$ (1). Let $a_1,...,a_n < b < a_{n+1},...,a_{n+k}$ be given where $a_i \in \cl A$ for each $i$ and $b\in B$. By the above lemma the element
$$
-a_1 \otimes \dot e_1 - \cdots  - a_n \otimes \dot e_n + a_{n+1} \otimes \dot e_{n+1} + \cdots + a_{n+k} \otimes \dot e_{n+k}
$$
must be strictly positive in $\cl B \otimes_{\max} \left(\ell_{n+k}^\infty/{\rm span}\{ (e_n, -e_k) \} \right)$. The order isomorphism guarantees that it must be positive in $\cl A \otimes_{\max} \left(\ell_{n+k}^\infty/{\rm span}\{ (e_n, -e_k) \} \right)$ too. By the same lemma, there exists a self-adjoint $a \in \cl A$ with the desired property. This finishes our proof.
\end{proof}

\begin{proof}[proof of Theorem 0.3] As $\ell_2^\infty \sqcup \ell_2^\infty$ is a C*-nuclear operator system (Proposition \ref{NP_2}) we have a canonical complete order embedding $\cl A \otimes_{\max} (\ell_2^\infty \sqcup \ell_2^\infty) \subseteq \cl B \otimes_{\max}(\ell_2^\infty \sqcup \ell_2^\infty)$. So by the above lemma $\cl A$ has TR(2,2)-property in $\cl B$.
\end{proof}

\begin{proof}[proof of Theorem 0.4] (1) $\Leftrightarrow$ (2). By Theorem \ref{main WRI}, $\cl A$ is w.r.i$.$ in $\cl B$ if and only if we have complete order embedding
$$
\cl A \otimes_{\max} \left( \ell_2^\infty \sqcup \ell_3^\infty \right) \subseteq 
\cl B \otimes_{\max} \left( \ell_2^\infty \sqcup \ell_3^\infty \right).
$$
With the identification $M_m(\cl A \otimes_{\max} \left( \ell_2^\infty \sqcup \ell_3^\infty \right)) \cong M_m(\cl A) \otimes_{\max} \left( \ell_2^\infty \sqcup \ell_3^\infty \right)$, and similarly for $\cl B$, this is equivalent to saying that we have an order embedding
$$
M_m(\cl A )\otimes_{\max} \left( \ell_2^\infty \sqcup \ell_3^\infty \right) \subseteq 
M_m(\cl B) \otimes_{\max} \left( \ell_2^\infty \sqcup \ell_3^\infty \right)
$$
for every $m$. By the above lemma this is equivalent to $M_m(\cl A)$ having TR(2,3)-property in $M_m(\cl B)$ for each $m$, or $\cl A$ having the complete TR(2,3)-property in $\cl B$.

\smallskip

The equivalence of  (1) and (3) can be shown identically. This completes our proof.
\end{proof}

We turn our attention to Connes' Embedding Problem on the embedding of II$_1$-factors.
Kirchberg proved, in his seminal paper \cite{Kirchberg94}, that this conjecture is equivalent to Kirchberg's Conjecture which states that every every C*-algebra with the local lifting property has the weak expectation property. His conjecture is equivalent to $C^*(\bb F_\infty)$ having weak expectation property. In Theorem 0.5 we express this problem in terms of Tight Riesz interpolation property and Arveson's extension property under certain order conditions on the morphism. We will start with the following observation:

\begin{proposition} Let $\cl S$ be an operator system. If $\cl S \otimes_{\min} C^*(\bb F_\infty) = \cl S \otimes_{\max} C^*(\bb F_\infty)$ holds order isomorphically  then the equality holds completely order isomorphically.
\end{proposition}

\begin{proof} We first fix a positive element $X\geq 0$ in $M_n(\cl S \otimes_{\min} C^*(\bb F_\infty) )$ that lies in the algebraic tensor product $M_n(\cl S \otimes C^*(\bb F_\infty) )$. (In fact we do not require completion of the tensor product.) We wish to show that $X$ is positive in $M_n(\cl S \otimes_{\max} C^*(\bb F_\infty) )$. We will use the identification $M_n(\cl S \otimes_{\min} C^*(\bb F_\infty) ) \cong \cl S \otimes_{\min} M_n(C^*(\bb F_\infty) )$ which also holds for the maximal tensor product. Let $\cl S_0 \subset \cl S$ and $\cl T_0 \subset M_n(C^*(\bb F_\infty) )$ be finite dimensional operator subsystems such that $X\in \cl S_0 \otimes \cl T_0$. By the injectivity of the min, $X$ is positive in $\cl S_0 \otimes_{\min} \cl T_0$. Let $\pi: C^* (\bb F_\infty) \rightarrow M_n(C^*(\bb F_\infty) )$ be a surjective unital $*$-homomorphism. Since $ M_n(C^*(\bb F_\infty) )$ has the local lifting property, there exists a ucp map $\gamma: \cl T_0 \rightarrow  C^*(\bb F_\infty) $ such that $\pi \circ \gamma$ coincides with the inclusion of $\cl T_0$ into $ M_n(C^*(\bb F_\infty) )$. Consider the maps
$$
\cl S_0 \otimes_{\min} \cl T_0 \xrightarrow{i \otimes \gamma} \cl S \otimes_{\min} C^*(\bb F_\infty) = 
\cl S \otimes_{\max} C^*(\bb F_\infty) \xrightarrow{\id \otimes \pi} \cl S \otimes_{\max} M_n(C^*(\bb F_\infty) ).
$$
Note that the first and and the last maps are completely positive and the middle equality holds order isomorphically. So the composition of these maps are positive which implies that $X\geq 0$ in $ \cl S \otimes_{\max} M_n(C^*(\bb F_\infty) ) \cong M_n(\cl S \otimes_{\max} C^*(\bb F_\infty) )$. Since the canonical map from $\cl S \otimes_{\max} C^*(\bb F_\infty)$ to $\cl S \otimes_{\min} C^*(\bb F_\infty)$ is readily completely positive the result follows.
\end{proof}

\begin{proof}[proof of Theorem 0.5] (1) $\Leftrightarrow$ (2). By Lemma \ref{TR & tensor}, $C^*(\bb F_\infty)$ has TR(2,3)-property in $B(\cl H)$ if and only if we have an order inclusion $$C^*(\bb F_\infty) \otimes_{\max} \left( \ell_2^\infty \sqcup \ell_3^\infty \right) \subseteq B(\cl H) \otimes_{\max} \left( \ell_2^\infty \sqcup \ell_3^\infty \right).$$
Since $\ell_2^\infty \sqcup \ell_3^\infty$ has the lifting property we have a canonical complete order isomorphism
$$
B(\cl H) \otimes_{\min} \left( \ell_2^\infty \sqcup \ell_3^\infty \right) = B(\cl H) \otimes_{\max} \left( \ell_2^\infty \sqcup \ell_3^\infty \right).
$$
As the minimal tensor product is injective the above inclusion is equivalent to the statement that we have an order isomorphism 
$$
C^*(\bb F_\infty) \otimes_{\min} \left( \ell_2^\infty \sqcup \ell_3^\infty \right)
=C^*(\bb F_\infty) \otimes_{\max} \left( \ell_2^\infty \sqcup \ell_3^\infty \right).
$$
By the above lemma this equality holds completely order isomorphically, equivalently, by Theorem 5.7 of \cite{kavruk2012}, $C^*(\bb F_\infty)$ has WEP, in other words, Connes' embedding problem has an affirmative solution. 

(1) $\Leftrightarrow$ (5). For an operator system $\cl S$ and positive linear functionals $\varphi_i \in \cl S^*$ satisfying $\varphi_1 + \varphi_2 = \varphi_3 + \varphi_4  = \varphi_5 + \varphi_6  $ there corresponds a positive map 
$$
\varphi: \cl S \rightarrow \sqcap_{i=1}^3 \ell_2^\infty  = 
\{ (\alpha_i)_{i=1}^6: \; \alpha_1 + \alpha_2 = \alpha_3 +\alpha_4 = \alpha_5 + \alpha_6\} \subset \ell^\infty_6
$$
given by $s \mapsto (\varphi_1(s),...,\varphi_6(s))$. Conversely, for every positive maps from $\cl S $ to $ \sqcap_{i=1}^3 \ell_2^\infty$ we obtain such positive linear functionals. Keeping this observation in mind, the condition in (5) is equivalent to the statement that every positive maps from  $C^*(\bb F_\infty) $ to $\sqcap_{i=1}^3 \ell_2^\infty$ extends a positive map on $B(\cl H)$. Since operator system dual of $\sqcap_{i=1}^3 \ell_2^\infty$ is completely order isomorphic to $\sqcup_{i=1}^3 \ell_2^\infty$, by the representation of the maximal tensor product, the last condition is equivalent to the statement that every positive linear functional on $C^*(\bb F_\infty) \otimes_{\max} \left(\sqcup_{i=1}^3\ell_2^\infty \right)$ extends to a positive linear functional on $B(\cl H) \otimes_{\max} \left(\sqcup_{i=1}^3\ell_2^\infty \right)$. This is same as saying that we have an order embedding
$$
C^*(\bb F_\infty) \otimes_{\max} \left(\sqcup_{i=1}^3\ell_2^\infty \right) \subseteq
B(\cl H) \otimes_{\max} \left(\sqcup_{i=1}^3\ell_2^\infty \right).
$$
Simply following the steps of the above proof (1) $\Leftrightarrow$ (2), one can show that the last statement is equivalent to Connes' embedding conjecture. 

\smallskip

 (3) $\Leftrightarrow$ (5). The equivalence of these two statements is similar to proof of Theorem 0.2 (equivalence of (2) and (4)) so we skip it.
 
 \smallskip
 
  (1) $\Leftrightarrow$ (4). Let $R$ be the quotient map from  $B(\cl H)^*$ to $C^*(\bb F_\infty)^*$ given by the restriction. The condition in (4) is equivalent to the statement that every cp map $\ell^\infty_2 \sqcup \ell^\infty_3$ to $C^*(\bb F^\infty)^*$ lifts to a cp map on  $B(\cl H)^*$:

$$
\xymatrix{
\ell^\infty_2 \sqcup \ell^\infty_3 \ar[rr]^{\gamma} \ar[rrd]_{\tilde \gamma} & & C^*(\bb F_\infty)^*\\
& & B(\cl H)^{*} \ar[u]_{R} 
}
$$
By the representation of the maximal tensor product this is equivalent to saying that every positive linear functionals on $\ell^\infty_2 \sqcup \ell^\infty_3 \otimes_{\max}  C^*(\bb F_\infty) $ extends to a positive linear functional on  $\ell^\infty_2 \sqcup \ell^\infty_3 \otimes_{\max}  B(\cl H) $. A moment of thought shows that this condition is equivalent to the statement that we have an order inclusion
$$
\ell^\infty_2 \sqcup \ell^\infty_3 \otimes_{\max}  C^*(\bb F_\infty) \subseteq \ell^\infty_2 \sqcup \ell^\infty_3 \otimes_{\max}  B(\cl H).
$$
Now as in the above proof, (1) $\Leftrightarrow$ (2), the last statement is equivalent to Connes' embedding conjecture.

\smallskip

(1) $\Rightarrow $ (6). In fact, if $\cl A \subseteq \cl B$ is w.r.i$.$ in $\cl B$, then for every finite dimensional operator system $\cl T$, every cp map $\varphi$ from $\cl A$ to $\cl T$ extends to cp map on $\cl B$. To see this one can first take the adjoint of $\varphi$ twice, $\varphi^{**} : \cl A^{**} \rightarrow \cl T^{**}\cong \cl T$, then compose it with the map extending the canonical inclusion $\cl A \hookrightarrow \cl A^{**}$ on $\cl B$. Keeping this in mind, assuming that Connes' embedding problem has an affirmative answer, that is, $C^{*}(\bb F_\infty)$ is w.r.i$.$ in $B(\cl H)$, every cp map $\gamma : C^{*}(\bb F_\infty) \rightarrow M_3 \sqcap M_3$ extends to a cp map from $B(\cl H)$ to $M_3 \sqcap M_3$, where the pullback $M_3 \sqcap M_3$ performed via tracial states. It is not hard to see that this coincides with the statement appears in (6).

\smallskip (6) $\Rightarrow $ (1). Consider $ \ell^{\infty}_3 \xrightarrow{i} M_3 \xrightarrow{p} \ell^{\infty}_3$, where $i$ is diagonal inclusion of $\ell^{\infty}_3$ into $M_3$ and $p$ is the unique conditional expectation from $M_3$ onto diagonals. The condition on $(6)$ is same as saying that every cp map from $C^{*}(\bb F_\infty)$ to $M_3 \sqcap M_3$ extends to a cp map on $B(\cl H)$ where the pullback is performed via tracial states. Now, by considering $i$ and $p$ above, this implies that every map from $C^{*}(\bb F_\infty)$ to $\ell^{\infty}_3 \sqcap \ell^{\infty}_3$ extends to a cp map from $B(\cl H)$ to $\ell^{\infty}_3 \sqcap \ell^{\infty}_3$, where the pullback is performed via the states $w(\alpha_1, \alpha_2, \alpha_3) = (\alpha_1+ \alpha_2+ \alpha_3)/3$. The predual of the operator system $\ell^{\infty}_3 \sqcap \ell^{\infty}_3$ is $\ell^{\infty}_3 \sqcup \ell^{\infty}_3$, therefore, by the representation of the maximal tensor product, this is same as saying that every positive linear functional on $\ell^\infty_3 \sqcup \ell^\infty_3 \otimes_{\max}  C^*(\bb F_\infty) $ extends to a positive linear functional on  $\ell^\infty_3 \sqcup \ell^\infty_3 \otimes_{\max}  B(\cl H) $. As above, this implies that we have an order inclusion
$$
\ell^\infty_3 \sqcup \ell^\infty_3 \otimes_{\max}  C^*(\bb F_\infty) \subseteq \ell^\infty_3 \sqcup \ell^\infty_3 \otimes_{\max}  B(\cl H).
$$
Now simply following the steps of the above proof, (1) $\Leftrightarrow$ (2), the last statement is equivalent to Connes' embedding conjecture. 
\end{proof}

$ $

In the last part of this section we work on C*TR($n,k$)-property then we discuss some of its consequences. We start with the following observation:
\begin{proposition}\label{System TR}
Let $\cl B$ be a C*-algebra, $\cl S \subseteq \cl B$ be an operator subsystem and C*$\{\cl S\}$ be the C*-algebra generated by $\cl S$. The following are equivalent:
\begin{enumerate}
\item $\cl S$ has C*TR($n,k$)-property in $\cl B$;
\smallskip
\item the operator system structure on the tensor product $\cl S \otimes \left( \ell_2^\infty \sqcup \ell_3^\infty \right)$ arising from the inclusion  $\cl B \otimes_{\max} \left( \ell_2^\infty \sqcup \ell_3^\infty \right)$, say $\tau_1$, coincides with the structure arising from the inclusion  $C^*\{\cl S\} \otimes_{\max} \left( \ell_2^\infty \sqcup \ell_3^\infty \right)$, say $\tau_2$, order isomorphically.
\end{enumerate}
\end{proposition}

\begin{proof} (1) $\Rightarrow$ (2). Note that the canonical map $C^*\{\cl S\} \otimes_{\max} \left( \ell_n^\infty \sqcup \ell_k^\infty \right) \rightarrow \cl B \otimes_{\max} \left( \ell_n^\infty \sqcup \ell_k^\infty \right)$ is completely positive. So its restriction $\cl S \otimes_{\tau_2} \left( \ell_n^\infty \sqcup \ell_k^\infty \right) \rightarrow \cl S \otimes_{\tau_1} \left( \ell_n^\infty \sqcup \ell_k^\infty \right)$ is readily a completely positive map. Therefore we need to prove that the inverse is positive. We use the identification $\ell_n^\infty \sqcup \ell_k^\infty \cong \ell^\infty_{n+k}/{\rm span}\{(e_n,-e_k)\}$. Let $\sum_{i=1}^{n+k} s_i \otimes \dot e_i$ be a strictly positive element in $ \cl S \otimes_{\tau_1} \ell^\infty_{n+k}/{\rm span}\{(e_n,-e_k)\}$. Here we may suppose that each $s_i$ is self-adjoint. By the definition of $\tau_1$, it is strictly positive in $ \cl B \otimes_{\max} \ell^\infty_{n+k}/{\rm span}\{(e_n,-e_k)\}$. By Lemma \ref{Lem TR char}, there exists a self-adjoint $b\in \cl B$ such that $-s_i<b<s_j$ for $i=1,...,n$ and $j= n+1,...,n+k$. Since $\cl S$ has C*TR($n,k$)-property in $\cl B$ there exists self-adjoint $a\in\,$C*$\{\cl S\}$ with the same property. Again, by Lemma \ref{Lem TR char},   $\sum s_i \otimes \dot e_i$ is strictly positive in $ C^*\{\cl S\} \otimes_{\max} \ell^\infty_{n+k}/{\rm span}\{(e_n,-e_k)\}$, so by the definition of $\tau_2$, it is strictly positive in $ \cl S \otimes_{\tau_2} \ell^\infty_{n+k}/{\rm span}\{(e_n,-e_k)\}$. This proves our claim.

\smallskip

(2) $\Rightarrow$ (1). Consider $s_1,...,s_n < b < s_{n+1},...,s_{n+k}$ where $s_i\in \cl S$ and $b\in \cl B$. We wish to show that there exists $a\in C^*\{\cl S\}$ with the same property. By Lemma \ref{Lem TR char}, $$
-s_1 \otimes \dot e_1 -\cdots - s_n \otimes \dot e_n + s_{n+1} \otimes \dot e_{n+1} + \cdots + s_{n+k} \otimes \dot e_{n+k}   
$$ is strictly positive in $ \cl B \otimes_{\max} \ell^\infty_{n+k}/{\rm span}\{(e_n,-e_k)\}$. So the definition of $\tau_1$ and the order isomorphism of $\tau_1$ and $\tau_2$ imply that the above element is strictly positive in
$$
\cl S \otimes_{\tau_2} \ell^\infty_{n+k}/{\rm span}\{(e_n,-e_k)\} \subseteq: C^*(\cl S) \otimes_{\max} \ell^\infty_{n+k}/{\rm span}\{(e_n,-e_k)\}.
$$
Thus, By Lemma \ref{Lem TR char}, there exists $a \in C^*\{\cl S\}$ with the desired property.
\end{proof}

\begin{corollary}
If $\cl S$ is a nuclear operator system then for every inclusion $\cl S \subseteq \cl B$, where $\cl B$ is a unital C*-algebra, $\cl S$ has the complete C*TR$(n,k)$-property in $\cl B$.
\end{corollary}

\begin{proof}[proof of Theorem 0.6] We let $\cl A$ denote the C*-algebra generated by $\cl S$ in $\cl B$. Assume first that $\cl S$ has the complete C*TR($n,k$)-property in $\cl B$.  By the above proposition we have a complete order isomorphism $\otimes_{\tau_1} = \otimes_{\tau_2}$ where
$$
\cl S \otimes_{\tau_1} \left( \ell_n^\infty \sqcup \ell_k^\infty  \right) \subseteq :
\cl B \otimes_{\max} \left( \ell_n^\infty \sqcup \ell_k^\infty  \right)
\subseteq \cl B \otimes_{\max} \left( \ell_n^\infty * \ell_k^\infty  \right)\,
$$
and
$$
\cl S \otimes_{\tau_2} \left( \ell_n^\infty \sqcup \ell_k^\infty  \right) \subseteq :
\cl A \otimes_{\max} \left( \ell_n^\infty \sqcup \ell_k^\infty  \right)
\subseteq \cl A \otimes_{\max} \left( \ell_n^\infty * \ell_k^\infty  \right).
$$
Here we observe that $\cl S \otimes_{\tau_2} \left( \ell_n^\infty \sqcup \ell_k^\infty  \right)$ must be a unitary operator subsystem of $ \cl A \otimes_{\max} \left( \ell_n^\infty * \ell_k^\infty  \right)$. We leave the elementary but tedious verification of this to the reader. Therefore, by Proposition 5.6 of \cite{kavruk2011}, $\cl A \otimes_{\max} \left( \ell_n^\infty * \ell_k^\infty  \right)$ must be the enveloping C*-algebra of $\cl S \otimes_{\tau_2} \left( \ell_n^\infty \sqcup \ell_k^\infty  \right)$. Now the rigidity of the enveloping C*-algebra ensures that the canonical unital $*$-homomorphism
$$
\cl A \otimes_{\max} \left( \ell_n^\infty * \ell_k^\infty  \right) \longrightarrow \cl B \otimes_{\max} \left( \ell_n^\infty * \ell_k^\infty  \right)
$$
is injective. Thus we have a canonical C*-algebra inclusion $\cl A \otimes_{\max} \left( \ell_n^\infty * \ell_k^\infty  \right) \subseteq  \cl B \otimes_{\max} \left( \ell_n^\infty * \ell_k^\infty  \right)$. Since $\ell_n^\infty \sqcup \ell_k^\infty  \subseteq \ell_n^\infty * \ell_k^\infty $  is an essential operator subsystem it follows that we have complete order inclusion $\cl A \otimes_{\max} \left( \ell_n^\infty \sqcup \ell_k^\infty \right) \subseteq \cl B \otimes_{\max} \left( \ell_n^\infty \sqcup \ell_k^\infty \right)$. So, by Lemma \ref{TR & tensor}, $\cl A$ has the complete TR($n,k$)-property in $\cl B$.

\smallskip

Conversely suppose that $\cl A$ has the complete TR($n,k$)-property in $\cl B$. By Lemma \ref{TR & tensor}, we have a complete order inclusion
$$
\cl A \otimes_{\max} \left( \ell_n^\infty \sqcup \ell_k^\infty \right) \subseteq \cl B \otimes_{\max} \left( \ell_n^\infty \sqcup \ell_k^\infty \right).
$$
It follows that the operator system structure on $\cl S \otimes_{} \left( \ell_n^\infty \sqcup \ell_k^\infty \right)$ arising from the inclusion $\cl B \otimes_{\max} \left( \ell_n^\infty \sqcup \ell_k^\infty \right)$ coincides with the structure arising from the inclusion $\cl A \otimes_{\max} \left( \ell_n^\infty \sqcup \ell_k^\infty \right)$ completely order isomorphically. So $\cl S$ has the complete C*TR($n,k$)-property in $\cl B$ by Proposition \ref{System TR}. 
\end{proof}

\begin{remark} In the statement of Theorem 0.6 with $\cl S \subseteq C^*\{\cl S\} \subseteq \cl B$ we require that $\cl S$ is a unitary operator subsystem of $C^*\{\cl S\}$. Without this condition we don't know any example of $\cl S \subseteq \cl B$ where $\cl S$ has a complete C*TR($n,k$)-property in $\cl B$ but $C^*\{\cl S\}$ does not have TR($n,k$)-property in $\cl B$. Also, if for every pair $\cl S \subseteq \cl B$ with $\cl S$ having complete C*TR($n,k$)-property in $\cl B$ if it is the case that $C^*\{\cl S\}$ has the TR($n,k$)-property in $\cl B$ then this solves Connes' Embedding Problem affirmatively. In fact we can fix a separable nuclear operator system $\cl S$ with the inclusions $\cl S \subseteq C^*_u(\cl S) \subseteq B(\cl H)$. Nuclearity assures that $\cl S$ has the complete C*TR(2,3)-property in $B(\cl H)$. So, with the previous assumption, $C_u^*(\cl S)$ has the complete TR(2,3)-property in $B(\cl H)$. This is same as saying that  $C_u^*(\cl S)$ is w.r.i$.$ in $B(\cl H)$, equivalently $C_u^*(\cl S)$ has WEP. Therefore, by Proposition 6.16 of \cite{kavruk2011}, Connes' Embedding Problem is affirmative. Finally it is worth mentioning that there exists operator systems such that for every embedding into a C*-algebra the only unitaries in the image are scalar multiples of the identity. To see this one can consider a universal operator system $\cl S$ in the sense of Kirchberg and Wassermann \cite{KW} and observe that the embedding of $\cl S$ into $C_u^*(\cl S)$ cannot hit any non-trivial unitary.
\end{remark}

\textbf{Acknowledgement.} This article is a part of study of relative weak injectivity for general operator systems. As the properties appear in the C*-algebra context fail for general operator systems we publish these results separately. In an upcoming paper we focus on w.r.i$.$ for operator systems and discuss weak rigitidy in this setting.

$ $

{\sc  Department of Mathematics \& Applied Mathematics,

Virginia Commonwealth University

Richmond, VA 23220, U.S.A.}

{\it E-mail address:} askavruk@vcu.edu

\end{document}